\documentclass{amsart}

\usepackage{amsmath,amssymb,amscd}
\usepackage{graphicx}
\usepackage[all]{xy}
\usepackage{epic}
\usepackage{caption,subcaption}
\usepackage{hyperref}
\input{hives.sty}

\newtheorem{theorem}{Theorem}[section]
\newtheorem{lemma}[theorem]{Lemma}
\newtheorem{corollary}[theorem]{Corollary}
\newtheorem{proposition}[theorem]{Proposition}

\theoremstyle{definition}
\newtheorem{definition}[theorem]{Definition}
\newtheorem{example}[theorem]{Example}

\theoremstyle{remark}
\newtheorem{remark}[theorem]{Remark}
\newtheorem{conjecture}[theorem]{Conjecture}
\newtheorem{problem}[theorem]{Problem}

\numberwithin{equation}{section}

\newcommand{\op}[1]{\operatorname{#1}}
\newcommand{\GL}{\operatorname{GL}}
\newcommand{\SL}{\operatorname{SL}}
\newcommand{\Gr}{\operatorname{Gr}}

\newcommand{\Rep}{\operatorname{Rep}}

\newcommand{\Hom}{\operatorname{Hom}}
\newcommand{\PHom}{\operatorname{PHom}}

\newcommand{\Ext}{\operatorname{Ext}}
\newcommand{\ext}{\operatorname{ext}}

\newcommand{\e}{{\sf e}}
\newcommand{\f}{{\sf f}}
\newcommand{\g}{{\sf g}}
\newcommand{\h}{{\sf h}}
\newcommand{\bn}{{\beta_n}}
\newcommand{\Tr}{\operatorname{Tr}}
\newcommand{\Ker}{\operatorname{Ker}}
\newcommand{\Coker}{\operatorname{Coker}}
\newcommand{\Img}{\operatorname{Im}}
\newcommand{\T}{\operatorname{T}}
\newcommand{\SI}{\operatorname{SI}}

\newcommand{\rank}{\operatorname{rank}}

\newcommand{\proj}{\operatorname{proj}\text{-}}

\newcommand{\dv}{\underline{\dim}}

\newcommand{\ckQ}{\widehat{k\Delta}}

\newcommand{\wtd}[1]{\widetilde{#1}}

\newcommand{\mc}[1]{\mathcal{#1}}
\newcommand{\mb}[1]{\mathbb{#1}}
\newcommand{\mr}[1]{{\sf #1}}%{\mathrm{#1}}
\newcommand{\bs}[1]{\boldsymbol{#1}}
\renewcommand{\b}[1]{\bold{#1}}
\newcommand{\cone}[1]{\mb{R}_+\Sigma_{\beta_{#1}}({T_{#1}})}
\newcommand{\br}[1]{\overline{#1}}
\newcommand{\innerprod}[1]{\langle#1\rangle}

\newcommand{\sm}[1]{\left(\begin{smallmatrix}#1\end{smallmatrix}\right)}

\newcommand{\fr}[1]{\framebox[1.2\width]{{$#1$}}}
\newcommand{\frs}[1]{\framebox[1.1\width]{{$\scriptscriptstyle #1$}}}
\def\lrcoef{c^\lambda_{\mu\,\nu}}
\def\tripar{(\lambda,\mu,\nu)}
\def\lr#1{{\sf LR}_{#1}}

\begin{document}

\title{Cluster Algebras and Semi-invariant Rings I. Triple Flags}
\author{Jiarui Fei}
\address{Department of Mathematics, University of California, Riverside, CA 92521, USA}
\email{jiarui@ucr.edu}
\thanks{}

%    General info
\subjclass[2010]{Primary 13F60, 16G20; Secondary 13A50, 52B20}

\date{}
\keywords{Semi-invariant Ring, Quiver Representation, Cluster Algebra, Quiver with Potential, Mutation, Cluster Character, Triple Flag, Stability, Littlewood-Richardson Coefficient, Hive Polytope, Hive Quiver, Lattice Point}

\begin{abstract}
We prove that each semi-invariant ring of the complete triple flag of length $n$ is an upper cluster algebra associated to an ice hive quiver.
We find a rational polyhedral cone $\mr{G}_n$ such that the generic cluster character maps its lattice points onto a basis of the upper cluster algebra.
%When $n<6$, these upper cluster algebras are cluster algebras.
As an application, we use the cluster algebra structure to find a special minimal set of generators for these semi-invariant rings when $n$ is small.
%Conversely, we find some properties of cluster algebras arises this ways using invariant theory.
\end{abstract}

\maketitle

\section*{Introduction}

The cluster algebra has established its wide connection with many areas in mathematics since its discovery by Fomin and Zelevinsky one and half decade ago. There are two mainstreams of results - one is categorifying cluster algebras, the other is finding cluster structure in natural mathematical objects. The purpose of this project is two-fold.
First, we hope to obtain more new examples, and put many old examples in a more uniform framework.
Second, we want to use the cluster algebra structure in an essential way to prove new results which seem unreachable by traditional methods.

The earliest examples of cluster algebras are the coordinate rings of $\SL_n/N$ and Grassmannians $\Gr\sm{n \\2}$ \cite{FZ1}.
Later the latter family was generalized by J. Scott to all Grassmannians of type $A$ \cite{Sc}, and by other authors to other types, remarkably \cite{GLSp}.
For quite a long time, they served as the main examples of cluster algebras.
Afterwards there are two kinds of generalizations, one is in a Lie-theoretic direction, and the other is in an invariant-theoretic direction.
In \cite{BFZ}, authors considered double Bruhat cells, and in \cite{GLSp}, authors considered unipotent cells and partial flag varieties.
Remarkably, J. Schr\"{o}er, B. Leclerc, and C. Greiss further generalized some of previous examples to the Kac-Moody setting \cite{GLSk}.
For the invariant-theoretic direction, S. Fomin and P. Pylyavskyy considered in \cite{FP} the semi-invariant ring of vectors and covectors in dimension three. It is further generalized to arbitrary dimension by K. Carde in a work-in-progress.

In this paper, we are going to see the cluster algebra structure in semi-invariant rings of quiver representation spaces.
For a fixed dimension vector $\beta$ of a quiver $Q$, the space of all $\beta$-dimensional representations is
$$\Rep_\beta(Q):=\bigoplus_{a\in Q_1}\Hom(k^{\beta(t(a))},k^{\beta(h(a))}).$$
The product of special linear group
$$\SL_\beta:=\prod_{v\in Q_0}\SL_{\beta(v)}$$
acts on $\Rep_\beta(Q)$ by the natural base change.
The rings of semi-invariants is by definition
$$\SI_\beta(Q):=k[\Rep_\beta(Q)]^{\SL_\beta}.$$
Many previous examples fall into this construction (up to some localization), including Grassmannians, partial flags, unipotent cells, double Bruhat cells, and mixed invariant rings at least in type $A$. So we just saw a tip of an iceberg.

However, it is impractical to consider any quiver with arbitrary dimension vector from the beginning.
The first family of quivers we consider here is the {\em triple flag quivers} $T_{p,q,r}$.
Such a family has established its importance around 2000 when Derksen and Weyman proved the saturation conjecture for the type-$A$ Littlewood-Richardson coefficients \cite{DW1}.
A triple flag $T_{p,q,r}$ with a dimension vector $\beta$ is called {\em complete} if $p=q=r=n$ and $\beta=\beta_n$ is {\em standard} as indicated below.
$$\vcenter{\xymatrix@R=1ex{
1\ar[r] & 2 \ar[r] & \cdots \ar[r]  & {n-1} \ar[dr] \\
1\ar[r] & 2 \ar[r] & \cdots \ar[r]  & {n-1} \ar[r] & n \\
1\ar[r] & 2 \ar[r] & \cdots \ar[r]  & {n-1} \ar[ur]
}}$$
We find quite amazingly that
\begin{theorem}[Theorem \ref{T:equal}]
$\SI_\bn(T_n)$ is the upper cluster algebra $\br{\mc{C}}(\Delta_n,\mc{S}_n)$ associated to the seed $(\Delta_n,\mc{S}_n)$.
\end{theorem}
\noindent Here, $\Delta_n$ is the ice {\em hive quiver} to be introduced in Section \ref{S:Hive}.
When $n=5$, it is displayed in Figure \eqref{fig:icehive}.
The initial {\em cluster variables} together with the {\em coefficients} \cite{FZ4} in $\mc{S}_n$ can be explicitly described in terms of {\em Schofield's semi-invariants} \cite{S1}.
By some rather trivial checking, we can show that the upper cluster algebra $\br{\mc{C}}(\Delta_n)$ is equal to the cluster algebra $\mc{C}(\Delta_n)$ for $n<6$.
However, we do not know if we always have the equality for any $n$.

%We first establish that $\SI_\bn(T_n)$ is the upper cluster algebra $\br{\mc{C}}(\Delta_n,\mc{S}_n)$,
%then the theorem follows easily from a result in \cite{MS}.
To show $\br{\mc{C}}(\Delta_n,\mc{S}_n)\subseteq \SI_\bn(T_n)$, we use an idea similar to \cite{FP}.
The hard part is to show the other containment, because we do not even know a finite set of generators of $\SI_\bn(T_n)$.
For this, we consider another well-known model for the (type-A) Littlewood-Richardson coefficients -- Knutson-Tao's hive model \cite{KT}.
In fact, we consider a variation provided by \cite{PV} for the sake of exposition.
We show that each {\em hive polytope} is (volume-preservingly) isomorphic to certain polytope $\sf{G}_n(\sigma)$ inside the rational convex cone $\sf{G}_n$ spanned by {\em $\mu$-supported $\g$-vectors}. As a consequence, the lattice points of $\sf{G}(\sigma)$ count the dimension of {\em $\sigma$-graded} piece of $\SI_\bn(T_n)$.

To finish the proof, we need help from the machinery of {\em quivers with potentials} \cite{DWZ1,DWZ2}.
We assign a {\em rigid} potential $W_n$ to $\Delta_n$ such that the associated {\em Jacobian algebra} $J_n:=J(\Delta_n,W_n)$ is finite-dimensional.
We can view the lattice points of $\sf{G}_n(\sigma)$ inside $K_0(\proj J_n)$, the Grothendieck group of the homotopy category of bounded complexes of projective representations of $J_n$.
The other containment follows from the fact that the generic character of P. Plamondon \cite{P} maps the lattice points of $\sf{G}_n(\sigma)$ bijectively to a linearly independent set in the upper cluster algebra. As a corollary of the proof, we construct a basis of $\br{\mc{C}}(\Delta_n)$.
This can be viewed as an algebraic analogue of Fock-Goncharnov conjecture for $\SL_n$ \cite{FG}.

\begin{theorem}[Corollary \ref{C:basis}] The generic character maps the lattice points of $\sf{G}_n$ bijectively to a basis of $\br{\mc{C}}(\Delta_n)$.
\end{theorem}

So far only in few cases, a basis of a (upper) cluster algebra with coefficients is known.
To author's best knowledge, they are bases for cluster algebras of surface type \cite{MSW}, and cluster algebras arising from unipotent cells \cite{GLSg}.
However, the former only deals with the principal coefficients, the latter only deals with the coefficient-free ones.
We point out that Y. Kimura and F. Qin constructed in \cite{KQ} a remarkable {\em positive} quantum basis in a setting similar to \cite{GLSg}.

Another reason why triple flags are important is that many quivers can be embedded in a categorical sense into the triple flags.
For example, the quiver $S_6$
$$\vcenter
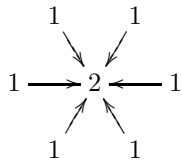
{\xymatrix@R=3ex@C=1ex{
& 1\ar[dr]  && 1 \ar[dl] \\
1 \ar[rr] && 2 && 1 \ar[ll] \\
& 1 \ar[ur]  && 1 \ar[ul]
}}$$
can be embedded into $T_4$ with the indicated dimension vector $\gamma$ corresponding to the standard one $\beta_4$.
So finding their cluster structure will be ultimately reduced to the case of triple flags \cite{Fs2}.
Note that $\SI_\gamma(S_6)$ is the coordinate ring of the Grassmannian $\Gr\binom{6}{2}$.
In this way, we expect to be able to recover a lot of previous known examples including unipotent cells, double Bruhat cells, Grassmannians, partial flags, invariants of vector and covectors, and so on \cite{Fs2}.

Except for getting new examples, we can use the cluster algebra structure to find generators and relations of semi-invariant rings.
This will be done in a more general framework in \cite{Fp}.
In the present paper, we find a special minimal set of generators of $\SI_\bn(T_n)$ for each $n<6$.
The author tried to prove the same result for $n=4$ using a geometric method similar to \cite{C}, but the proof is quite long \cite{Ft4}.
The method seems hard to be generalized to bigger $n$. Now with the help of the cluster structure, we can finish all proofs in few pages.

This paper is organized as follows.
In Section \ref{S:SI}, we recall the work of Schofield, Derksen-Weyman, etc., on the semi-invariant rings of quiver representations.
In Section \ref{S:Tn}, we specialize the general theory to the family of triple flag quivers and make connection with the Littlewood-Richardson coefficients following \cite{DW1,DW2}.
In Section \ref{S:CA}, we recall the definition of cluster algebras and their upper bounds.
We introduce the weight configuration in Definition \ref{D:wtconfig}, which serves as the first layer of the possible cluster structure.

In Section \ref{S:QP}, we recall the mutation of quivers with potentials and their representations following \cite{DWZ1,DWZ2}.
Since we need to deal with general coefficient system, we consider something slightly more general,
called ice quivers with potentials and their $\mu$-supported representations (Definition \ref{D:ice}, \ref{D:mu_supported}).
In Section \ref{S:CC}, we recall the cluster character $C$ date back to Caldero-Chapton.
The version that we consider is in the setting of \cite{DWZ1,DWZ2}.
Theorem \ref{T:CC} gives a sufficient condition for the image of a set of representations being a linearly independent set in the upper cluster algebra.
Specializing this result to the generic character, we get a similar result in \cite{P}.
The key difference is that our domain $G(Q,W)$ is a subset of the full lattice $K_0(\proj J)$ due to the difference in the definition of upper cluster algebras.

In Section \ref{S:Hive}, we introduce the ice hive quiver with potential $(\Delta_n,W_n)$.
We give in Theorem \ref{T:LP_Gn} a complete description of the domain $G(\Delta_n,W_n)$ as lattice points of certain rational polyhedral cone $\mr{G}_n$.
In Section \ref{S:LR}, we prove in Theorem \ref{T:VPiso} that there is a volume-preserving linear transformation mapping the
cone $\mr{G}_n$ onto the Littlewood-Richardson cone.

In Section \ref{S:CS}, we prove our main results. We establish the cluster structure of $\SI_\bn(T_n)$ in Theorem \ref{T:equal}.
One difficult step is to show that our chosen initial seed is algebraically independent (Theorem \ref{T:ag-independent}).
As a corollary, we construct a basis of these (upper) cluster algebras.
In Section \ref{S:smalln}, we use the cluster structure to find a special minimal set of generators of $\SI_\bn(T_n)$ for $n< 6$.
%In the last section, we list a few interesting open problems to the interested readers.

\subsection*{Notations and Conventions}
Our vectors are exclusively row vectors. All modules are right modules.
For a quiver $Q$, we denote by $Q_0$ the set of vertices and by $Q_1$ the set of arrows.
For an arrow $a$, we denote by $t(a)$ and $h(a)$ its tail and head.
Arrows are composed from left to right, i.e., $ab$ is the path $\cdot \xrightarrow{a}\cdot \xrightarrow{b} \cdot$.
Throughout the paper, the base field $k$ is algebraically closed of characteristic zero.
Unadorned $\Hom$ and $\otimes$ are all over the base field $k$, and the superscript $*$ is the trivial dual.
For any representation $M$, $\dv M$ is the dimension vector of $M$.
For direct sum of $n$ copies of $M$, we write $nM$ instead of the traditional $M^{\oplus n}$.

\section{Semi-invariants of Quiver Representations} \label{S:SI}
\subsection{Schofield's Construction}
Let us briefly recall the semi-invariant rings of quiver representations \cite{S1}.
Let $Q$ be a finite quiver without oriented cycles.
We fix an algebraically closed field $k$ of characteristic zero.
For a dimension vector $\beta$ of $Q$, let $V$ be a $\beta$-dimensional vector space $\prod_{i\in Q_0} k^{\beta(i)}$. We write $V_i$ for the $i$-th component of $V$.
The space of all $\beta$-dimensional representations is
$$\Rep_\beta(Q):=\bigoplus_{a\in Q_1}\Hom(V_{t(a)},V_{h(a)}).$$
The product of general linear group
$$\GL_\beta:=\prod_{i\in Q_0}\GL(V_i)$$
acts on $\Rep_\beta(Q)$ by the natural base change. This action has a {\em kernel}, which is the multi-diagonally embedded multiplicative group $k^*$.
Define $\SL_\beta\subset \GL_\beta$ by
$$\SL_\beta=\prod_{i\in Q_0}\SL(V_i).$$
We are interested in the rings of semi-invariants
$$\SI_\beta(Q):=k[\Rep_\beta(Q)]^{\SL_\beta}.$$
The ring $\SI_\beta(Q)$ has a weight space decomposition
$$\SI_\beta(Q)=\bigoplus_\sigma \SI_\beta(Q)_\sigma,$$
where $\sigma$ runs through the multiplicative {\em characters} of $\GL_\beta$.
We refer to such a decomposition the $\sigma$-grading of $\SI_\beta(Q)$.
Recall that any character $\sigma: \GL_\beta\to k^*$ can be identified with a weight vector
$\sigma \in \mb{Z}^{Q_0}$
\begin{equation} \label{eq:char} \big(g(i)\big)_{i\in Q_0}\mapsto\prod_{i\in Q_0} \big(\det g(i)\big)^{\sigma(i)},
\end{equation}
Since $Q$ has no oriented cycles, the degree zero component is the field $k$ \cite{Ki}.

Let us understand these multihomogeneous components
$$\SI_\beta(Q)_\sigma:=\{f\in k[\Rep_\beta(Q)]\mid g(f)=\sigma(g)f, \forall g\in\GL_\beta \}.$$
For any projective presentation $f: P_1\to P_0$, we view it as an element in the homotopy category $K^b(\proj Q)$ of bounded complexes of projective representations of $Q$.
The {\em weight vector} $\f$ of $f$ is the corresponding element in the Grothendieck group of $K^b(\proj Q)$.
Concretely, suppose that $P_1=P(\sigma_1)$ and $P_0=P(\sigma_0)$ for $\sigma_i\in\mathbb{N}_0^{Q_0}$, then $\f = \sigma_1-\sigma_0$.
Here, we use the notation $P(\sigma)$ for $\bigoplus_{i\in Q_0} \sigma(i) P_i^{}$, where $P_i$ is the indecomposable projective representation corresponding to the vertex $i$.
From now on, we will view a weight $\sigma$ as an element in the dual $\Hom_{\mb{Z}}(\mb{Z}^{Q_0},\mb{Z})$ via the usual dot product.

We set $\sigma=\f$, and assume that $\sigma(\beta)=0$.
We apply the functor $\Hom_Q(-,N)$ to $f$ for $N\in\Rep_\beta(Q)$
\begin{equation} \label{eq:canseq} \Hom_Q(P_0,N)\xrightarrow{\Hom_Q(f,N)}\Hom_Q(P_1,N).
\end{equation}
Since $\sigma(\beta)=0$, $\Hom_Q(f,N)$ is a square matrix.
Following Schofield \cite{S1}, we define
$$s(f,N):=\det \Hom_Q(f,N).$$
We give a more concrete description for the map $\Hom_Q(f,N)$.\\
{\bf Concrete description of $\Hom_Q(f,N):\ $}
Recall that a morphism $P_1\xrightarrow{f} P_0$ can be represented by a matrix whose entries are linear combination of paths. Applying $\Hom_Q(-,N)$ to this morphism is equivalent to that we first transpose the matrix of $f$, then substitute paths in the matrix by corresponding matrix representations in $N$.

We set $s(f)(-)=s(f,-)$ as a function on $\Rep_\beta(Q)$. It is proved in \cite{S1} that $s(f)\in\SI_\beta(Q)_{\sigma}$.
%$c_N\in\SI_\alpha^{\sigma_\beta^\vee}(Q)$ for $\sigma_\beta^\vee=-\innerprod{-,\beta}_Q$.
In fact,
\begin{theorem}[\cite{DW1,SV}] \label{T:inv_span} $s(f)$'s span $\SI_\beta(Q)_{\sigma}$ over the base field $k$.
\end{theorem}

It is easy to see that if $s(f)\neq 0$, then $f$ resolves some representation $M$, say of dimension $\alpha$
$$0\to P_1 \xrightarrow{f} P_0\to M\to 0.$$
From the long exact sequence
\begin{equation} \label{eq:canseq} \Hom_Q(M,N)\hookrightarrow\Hom_Q(P_0,N)\xrightarrow{\Hom_Q(f,N)}\Hom_Q(P_1,N)\twoheadrightarrow\Ext_Q(M,N),
\end{equation}
we see that $\alpha$ and $\sigma$ are related by
$\sigma(-)=-\innerprod{\alpha,-}_Q$, where $\innerprod{-,-}_Q$ is the Euler form of $Q$.
In this case, we call $\sigma$ the weight vector corresponding to $\alpha$, and denote it by $ ^\sigma\!\alpha$;
and conversely we call $\alpha$ the dimension vector corresponding to $\sigma$, and denote it by $^\alpha\!\sigma$.
It also follows from \eqref{eq:canseq} that $s(f,N)\neq 0$ if and only if $\Hom_Q(M,N)=0$, or equivalently~$\Ext_Q(M,N)=0$.
In this case, we say $M$ is (left) {\em orthogonal} to $N$, denoted by $M\perp N$.

We want to point out that the function $s(f)$ is determined, up to a scalar multiple, by the homotopy equivalent class of $f$, and thus by the isomorphism class of $M$.
If one hopes to define the function $s(f)$ uniquely in terms of representations, one can take the {\em canonical resolution} $f_{\op{can}}$ of $M$. In this case, we denote $s_M:=s(f_{\op{can}})$.
This is Schofield's original definition in \cite{S1}.

\subsection{Stability} We define the subgroup $\GL_\beta^\sigma$ to be the kernel of the character map \eqref{eq:char}. The invariant ring of its action is $$\SI_\beta^\sigma(Q):=k[\Rep_\beta(Q)]^{\GL_\beta^\sigma}=\bigoplus_{n\geqslant 0} \SI_\beta(Q)_{n\sigma}.$$

\begin{definition}
A representation $M\in\Rep_\beta(Q)$ is called {\em $\sigma$-semi-stable}
if there is some non-constant $f\in \SI_\beta^\sigma(Q)$ such that $f(M)\neq 0$.
It is called {\em stable} if the orbit $\GL_\beta^\sigma\cdot M$ is closed of dimension equal to $\dim\GL_\beta^\sigma-1$.
%We denote the set of all $\sigma$-semi-stable (resp. $\sigma$-stable, $\sigma$-unstable) representations in $\Rep_\alpha(Q)$ by $\Rep_\alpha\ss{}{ss}(Q)$ (resp. $\Rep_\alpha\ss{}{st}(Q)$, $\Rep_\alpha\ss{}{un}(Q)$). When $\Rep_\alpha\ss{}{ss}(Q)$ (resp. $\Rep_\alpha\ss{}{st}(Q)$) is non-empty, we say that $\alpha$ is $\sigma$-semi-stable (resp. $\sigma$-stable).
\end{definition}

\noindent Based on Hilbert-Mumford criterion, King provided a simple criterion for the stability of a representation.
\begin{lemma} \cite[Proposition 3.1]{Ki}  \label{L:King} A representation $M$ is $\sigma$-semi-stable (resp. $\sigma$-stable) if and only if $\sigma(\dv M)=0$ and $\sigma(\dv L)\geqslant 0$ (resp. $\sigma(\dv L)>0$) for any non-trivial subrepresentation $L$ of $M$.
\end{lemma}

Let $\Sigma_\beta(Q)$ be the set of all weights $\sigma$ such that $\SI_\beta(Q)_\sigma$ is non-empty.
Since $\SI_\beta(Q)_{\bs{0}}=k$, it spans a {\em pointed} cone $\mb{R}^+\Sigma_\beta(Q)$.
The next theorem is an easy consequence of King's stability criterion and Theorem \ref{T:inv_span}.
By a general $\beta$-dimensional representation, we mean in a sufficiently small Zariski open subset (``sufficient" here depends on the context).
Following \cite{S2}, we use the notation $\gamma\hookrightarrow\beta$ to mean that
a general $\beta$-dimensional representation has a $\gamma$-dimensional subrepresentation.

\begin{theorem} \cite[Theorem 3]{DW1} \label{T:cone} We have
$$\Sigma_\beta(Q)=\{\sigma\in \Hom_{\mb{Z}}(\mb{Z}^{Q_0},\mb{Z}) \mid \sigma(\beta)=0 \text{ and } \sigma(\gamma)\geq 0 \text{ for all } \gamma\hookrightarrow \beta\}.$$
In particular, $\Sigma_\beta(Q)$ is a saturated semigroup.
\end{theorem}

Schofield relates the condition $\gamma\hookrightarrow \beta$ to some generic homological condition.
Following \cite{S2}, we introduce the notation $\hom_{Q}(\alpha,\beta)$ (resp. $\ext_{Q}(\alpha,\beta)$) to denote the dimension of the space of homomorphisms (resp. extensions) from a general $\alpha$-dimensional representation to a general $\beta$-dimensional representation of $Q$.
Let
$$\Rep_{\gamma\hookrightarrow \beta}(Q)=\{ M\in\Rep_\beta(Q)\mid M \text{ has a $\gamma$-dimensional subrepresentation} \}.$$
This is a closed subvariety of $\Rep_\beta(Q)$.
\begin{lemma}[{\cite[3]{S2}}] \label{L:ext} The codimension of $\Rep_{\gamma\hookrightarrow \beta}(Q)$ is equal to $\ext_Q(\gamma,\beta-\gamma)$.
\end{lemma}

\begin{lemma}\cite[Lemma 1]{DW1} \label{L:exact} Suppose that we have an exact sequence of representations of $Q$
$$0\to L \to M \to N\to 0$$
with $\innerprod{\dv L,\beta}=\innerprod{\dv N,\beta}=0$, then as a function on $\Rep_\beta(Q)$, $s(M)$ is, up to a scalar, equal to $s(L)s(N)$.
\end{lemma}

Recall that a dimension vector $\alpha$ is called a {\em real} root if $\innerprod{\alpha,\alpha}_Q=1$.
It is called {\em Schur} if $\Hom_Q(M,M)=k$ for
$M$ general in $\Rep_\alpha(Q)$. For a real Schur root $\alpha$, $\Rep_{n\alpha}(Q)$ has a dense orbit for any $n\in \mb{N}$.

\begin{lemma} \label{L:dim=1}
If $\alpha$ is a real Schur root, then $\dim\SI_\beta(Q)_{-\innerprod{n\alpha,-}}=1$ for any $n\in \mb{N}$.
\end{lemma}

\begin{proof} As an easy consequence of Theorem \ref{T:inv_span}, we have the following reciprocity property (\cite[Corollary 1]{DW1})
$$\dim \SI_\beta(Q)_{-\innerprod{\alpha,-}}=\dim \SI_\alpha(Q)_{\innerprod{-, \beta}}.$$
Since $\Rep_{n\alpha}(Q)$ has a dense orbit, the ring of rational invariants $k(\Rep_{n\alpha}(Q))^{\GL_{n\alpha}}$ is trivial.
So if $f,g\in \SI_{n\alpha}(Q)_{\innerprod{-,\beta}}$, then $f/g\in k(\Rep_{n\alpha}(Q))^{\GL_{n\alpha}}=k$.
The result follows from the reciprocity property.
\end{proof}

A weight is called {\em extremal} in $\Sigma_\beta(Q)$ if it lies on an extremal ray of $\mb{R}_+\Sigma_\beta(Q)$.
An indivisible extremal weight $\sigma$ must be indecomposable in $\Sigma_\beta(Q)$, i.e., it cannot be written as
$\sigma=\sigma_1+\sigma_2$ with $\sigma_i\in \Sigma_\beta(Q)$.

\begin{lemma} \label{L:irreducible} If $\sigma$ is an indivisible extremal weight, then any semi-invariant function $s$ of weight $\sigma$ is an irreducible polynomial.
\end{lemma}

\begin{proof} Suppose that $\sigma$ spans an extremal ray $\sf{r}$.
Since the cone $\mb{R}^+\Sigma_\beta(Q)$ is pointed and $\sf{r}$ is extremal, there is a set of dimension vectors $\{\gamma_i\}_i$ such that $\sf{r}$ is defined by the hyperplanes $\innerprod{-,\gamma_i}=0$.
Moreover, $\theta(\gamma_i)\geq 0$ for all $\theta\in \Sigma_\beta(Q)$ and all $i$, and $\theta(\gamma_i)>0$ for some $i$ if $\theta \notin \sf{r}$.
We define a total degree $d$ on the semi-invariant ring by setting $d(\sigma) = \sum_i \sigma(\gamma_i)\geq 0$. Note that the total degree $0$ component is $k \oplus \SI_\beta^\sigma (Q)$.

Suppose that $s$ factors as $s=s_1s_2$, then both $s_1$ and $s_2$ must be homogeneous of total degree~$0$.
Since $\sigma$ is indivisible, it is clear that one of them has to be a unit.
\end{proof}

%\section{General Representations and Some Algorithms}
%
%\subsection{Computing Extensions} CE
%\subsection{Canonical Decompositions} CD
%\subsection{Orthogonal Projections} OP
%\subsection{Finding Extremal Rays} ER

\section{The Triple Flag Quivers} \label{S:Tn}
In this section, we specialize the previous discussion to the case of triple flag quivers, and make connection with the Littlewood-Richardson coefficients following \cite{DW1,DW2}.

Let $T_{p,q,r}$ be the quiver with $p+q+r-2$ vertices:
$$\vcenter{\xymatrix@=4ex{
^11 \ar[r] & ^12 \ar[r] & \cdots \ar[r] & ^1{p-2} \ar[r] & ^1{p-1} \ar[dr] \\
^21\ar[r] & ^22 \ar[r] & \cdots \ar[r] & ^2{q-2} \ar[r] & ^2{q-1} \ar[r] & ^3r \\
^31\ar[r] & ^32 \ar[r] & \cdots \ar[r] & ^3{r-2} \ar[r] & ^3{r-1} \ar[ur]
}}$$
We use the convention ${^1p}={^2q}={^3r}.$ We know from \cite{DW1} that if we take the {\em standard} dimension vector $\bn=\sm{1 & 2 & \cdots & {n-1}\\ 1 & 2 & \cdots & {n-1} & n, \\ 1 & 2 & \cdots & {n-1} }$ for $T_n:=T_{n,n,n}$,
then we can view $\dim\SI_\bn(T_n)_\sigma$ as a type-A Littlewood-Richardson coefficient.

\begin{theorem}[{\cite[7.1]{DW2}}] \label{T:eqLR}
If $\sigma$ is given by
$\sigma=\sm{a_1 & a_2 & \cdots & a_{n-1}\\ b_1 & b_2 & \cdots & b_{n-1} & c_n, \\ c_1 & c_2 & \cdots & c_{n-1} }$, then
\begin{equation*} \label{eq:LR} \dim\SI_\bn(T_n)_\sigma=c_{\mu,\nu}^{\lambda},
\end{equation*}
where \begin{align}
\label{eq:wt2par1} &\mu = \mu(\sigma) = (a_1+\cdots+a_{n-1},a_2+\cdots+a_{n-1},\cdots,a_{n-1}),\\
\label{eq:wt2par2} &\nu = \nu(\sigma) =
(b_1+\cdots+b_{n-1},b_2+\cdots+b_{n-1},\cdots,b_{n-1}),\\
\label{eq:wt2par3} &\lambda = \lambda(\sigma) = (-c_n,-(c_n+c_{n-1}),\cdots,-(c_n+c_{n-1}+\cdots+c_1)).
\end{align}
\end{theorem}
%\noindent We denote the map $\sigma\mapsto (\lambda(\sigma),\mu(\sigma),\nu(\sigma))$ by $\varpi: \mb{Z}^{(T_n)_0} \to \mb{Z}^{3n-2}.$
%Conversely, if $\lambda,\mu,\nu\in\mb{Z}^n$, then
%$$\lrcoef=\dim\SI_\bn^\sigma(Q),$$
%where
%\begin{align*}
%\end{align*}

Note that for $T_n$ a Schur root either has support on one arm (in which case it corresponds to a positive root of $A_n$), or it is nondecreasing along each arm. A {\em facet} of the cone $\cone{n}$ is a codimension one face of $\cone{n}$. A weight in $\Sigma_\bn(T_n)$ is called {\em extremal} if it lies on an extremal ray (dimension one face) of $\cone{n}$.
\begin{theorem}[{\cite[Theorem 7.8]{DW2}}]  \label{T:CT}
For every pair $(\gamma,\bn)$ with $\bn=\gamma+\beta_n'$, $\gamma,\beta_n'$ nondecreasing along arms, and $\gamma\circ \beta_n'=1$, the inequality $\sigma(\gamma)\geq 0$ defines a facet of $\cone{n}$. All nontrivial facets can be uniquely obtained this way.
\end{theorem}
\noindent Here, $\gamma\circ \beta_n'$ is the number of $\gamma$-dimensional subrepresentations of a general $\bn=\gamma+\beta_n'$ dimensional representation. It can also be interpreted as a Littlewood-Richardson coefficient (see \cite[Section 7]{DW2} for detail). Theorem \ref{T:CT} has been implemented as an algorithm to compute the cone $\cone{n}$. We call it Algorithm CT.

Let $\e_i^1,\e_j^2,\e_k^3$ be the unit vectors supported on the vertex ${^1i},{^2j},{^3k}$ respectively, so
$\{\e_i^a,\e_n \}_{i=1\dots n-1}^{a=1,2,3}$ form a standard basis of $\mb{Z}^{(T_n)_0}$.
We will use the convention that $\e_n^a=\e_n$.
It follows from Theorem \ref{T:cone} that if $\sigma\in \Sigma_\bn(T_n)$, then the last coordinate $\sigma(n)$ of $\sigma$ must be positive.
The {\em level} of $\sigma$ is by definition $\sigma(n)$.
\begin{lemma} \label{L:level1}
Each level-1 weight $\sigma\in \Sigma_\bn(T_n)$ is extremal.
It is of one of the following forms
\begin{align*}
\tag{I} &\f_{i,j}^{a,b}:=\e_n-\e_i^a-\e_j^b, & &i+j=n, a\neq b; \\
\tag{II} &\f_{i,j}:=\e_n-\e_i^1-\e_j^2-\e_k^3, & &i+j+k=n.
\end{align*}
Its corresponding dimension vector $^\alpha\!\sigma$ is a real Schur root.
There are exactly~$3(n-1)$ of first type, and $\sm{n-1\\2}$ of the second type.
\end{lemma}

\begin{proof} Since $\sigma(n)$ must be positive, each level-1 $\sigma\in \Sigma_\bn(T_n)$ is extremal.
The conditions $i+j=n$ and $i+j+k=n$ are clearly necessary.
Suppose that $\sigma=\e_n-\e_i^a-\e_j^a-\e_k^b$ with $i\leq j$, then
$\gamma=\sum_{l=i}^{n}\e_l^a$ is a dimension vector contradicting Theorem \ref{T:cone}.
So the superscripts of terms in $\sigma$ must be all different.
We can also use Lemma \ref{L:level1_par} below to check that the above $\f_{i,j}^{a,b}$ and $\f_{i,j}$ do lie in $\Sigma_\bn(T_n)$ because
$c_{1^i,1^j}^{1^{i+j}}=1$.
To show $^\alpha\!\sigma$ is real Schur, we check that $\innerprod{^\alpha\!\sigma,{^\alpha\!\sigma}}=1$, and a general representation of dimension $^\alpha\!\sigma$ is indecomposable. This is quite obvious.
\end{proof}

\begin{remark} We observe that $\sm{n-1\\2}$ is exactly the dimension of GIT quotient of $\Rep_\bn(T_n)$ for a generic stability, and $\sm{n+2\\2}-3=\sm{n-1\\2}+3(n-1)$ is the Krull dimension of $\SI_\bn(T_n)$.
Later we will see that the first type are related to coefficient variables, and the second type are related to (initial) cluster variables.
\end{remark}

The author do not like the notation $\f_{i,j}^{a,b}$, and want to treat two types uniformly.
So we introduce the convention that $\e_0^a$ is the zero vector for $a=1,2,3$, then
\begin{equation} \label{eq:fij_boundary} \f_{i,0}=\f_{i,n-i}^{1,3},\ \f_{0,j}=\f_{j,n-j}^{2,3},\ \text{ and  } \f_{i,j}=\f_{i,j}^{1,2}\ (\text{if } i+j=n).
\end{equation}
Recall the correspondence \eqref{eq:wt2par1}--\eqref{eq:wt2par3}. It is easy to verify that
\begin{lemma} \label{L:level1_par}
The level-1 weights $\f_{i,j}$ are precisely mapped to the partitions $(1^{i+j},1^i,1^j)$. % with $0\leq i,j\leq n-1$.
\end{lemma}

For any $\f_{i,j}$, we can associate a projective presentation
\begin{align}
%& f_{i,j}^{a,b}: P_i^a \oplus P_j^b \xrightarrow{( p_i^a\  p_j^b)^{\T}} P_n,\\
& f_{i,j}: P_i^1 \oplus P_j^2 \oplus P_k^3 \xrightarrow{( p_i^1\  p_j^2\ p_k^3)^{\T}} P_n.
\end{align}
Here, we use the convention that $P_0^a=0$, and $p_i^a$ is the unique path from $^a i$ to $n$.
By Lemma \ref{L:level1} and \ref{L:dim=1}, up to a scalar multiple the element in $\SI_\bn(T_n)_{\f_{i,j}}$ is equal to the Schofield's semi-invariant function $s(f_{i,j})$.
We will see in Section \ref{S:CS} that all $s(f_{i,j})$'s are algebraically independent over $k$.

\begin{example} \label{Ex:T4} Consider the quiver $T_4$ with the standard $\beta_4$.
We run Algorithm CT, and find 18 extremal rays in $\cone{4}$. Besides the 12 level ones, they are
\begin{align*}
2\e_n-\e_1^a-\e_2^b-\e_2^c-\e_3^a,\quad \text{and}\quad 2\e_n-\e_2^a-\e_3^b-\e_3^c.  %\f^{a}  %\f^{a'}
\end{align*}
We leave for readers to check that both correspond to real Schur roots.
%So up to a scalar multiple, they correspond to the functions
%$s(f^a)$ and $s(f^{a'})$, where $f^a$ and $f^a$ are the following projective presentation
%\begin{align*}
%& P_{1}^a \oplus P_2^b \oplus P_2^c \oplus P_{3}^a \xrightarrow{\sm{p_1^a & p_{2}^b & 0 & p_{3}^a \\0 & 0 & p_2^c & p_{3}^a}^{\T}} 2P_n, \\
%& P_2^a \oplus P_3^b \oplus P_{3}^c \xrightarrow{\sm{p_2^a & 0 & p_{3}^c\\0 & p_3^b & p_{3}^c}^{\T}} 2P_n.
%\end{align*}
We give a concrete description for some $s(f_{i,j})$. For example,
$$s(f_{01})=\det\begin{pmatrix} B_1B_2B_3\\ C_3\end{pmatrix}, \qquad s(f_{11})=\det\begin{pmatrix} A_1A_2A_3 \\ B_1B_2B_3 \\ C_2C_3 \end{pmatrix},$$
where $A_i,B_i,C_i$ are generic matrices as shown below.
$$\vcenter
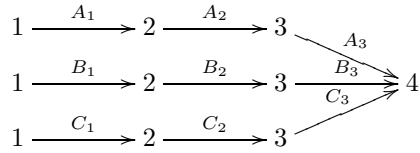
{\xymatrix@R=2ex@C=9ex{
1\ar[r]^{A_1} & 2 \ar[r]^{A_2} & 3 \ar[dr]^{A_3}  \\
1\ar[r]^{B_1} & 2 \ar[r]^{B_2} & 3 \ar[r]^{B_3}   & 4 \\
1\ar[r]^{C_1} & 2 \ar[r]^{C_2} & 3 \ar[ur]^{C_3}
}}$$

\end{example}

\section{Graded Cluster Algebras} \label{S:CA}
\subsection{Cluster Algebras}
We follow mostly Section 3 of \cite{FP}.
The combinatorial data defining a cluster algebra is encoded in an {\em ice} quiver $\Delta$ with no loops or oriented 2-cycles.
The first $p$ vertices of $\Delta$ are designated as {\em mutable}; the remaining $q-p$ vertices are called {\em frozen}.
If we require no arrows between frozen vertices, then
such a quiver is uniquely determined by its {\em $B$-matrix} $B(\Delta)$.
It is a $p\times q$ matrix given by
$$b_{u,v} = |\text{arrows }u\to v| - |\text{arrows }v \to u|.$$
%To define the cluster algebras, we need the notion of quiver mutations.

\begin{definition} \label{D:Qmu}%[\emph{Quiver mutations}]
Let $u$ be a mutable vertex of $\Delta$.
The {\em quiver mutation} $\mu_u$ transforms $\Delta$ into the new quiver $\Delta'=\mu_u(\Delta)$ via a sequence of three steps.
\begin{enumerate}
\item For each pair of arrows $v\to u\to w$, introduce a new arrow $v\to w$ (unless both $v$ and $w$ are frozen, in which case do nothing);
\item Reverse the direction of all arrows incident to $u$;
\item Remove all oriented 2-cycles.
\end{enumerate}
\end{definition}

\noindent The above recipe can be reformulated in terms of $B$-matrix as follows.
Let $\phi$ be the $q\times q$ matrix obtained from the identity matrix by replacing the $u$-th row by a vector $\phi_u$ where $$\phi_u(u)=-1,\quad \text{and}\quad \phi_u(v)=|\text{arrows } v\to u |.$$
We write $\phi_{\op{p}}$ for the restriction of $\phi$ to its $p\times p$ upper left corner.
Then the mutated $B$-matrix $B'$ for $\Delta'$ is related to the original one by
\begin{equation} \label{eq:mu_B} B'=\phi_{\op{p}}^{\T} B\phi.
\end{equation}
We note that $\phi=\phi^{-1}$ so the quiver mutation is an {\em involution}.

Quiver mutations can be iterated {\em ad infinitum}, using an arbitrary sequence of mutable vertices of an evolving quiver. This combinatorial dynamics drives the algebraic dynamics of seed mutations that we describe next.

\begin{definition} \label{D:seeds} %[\emph{Seeds and their mutations}]
Let $\mc{F}$ be a field containing $k$.
A {\em seed} in $\mc{F}$ is a pair $(\Delta,\b{x})$ consisting of an ice quiver $\Delta$ as above together with a collection $\b{x}=\{x_1,x_2,\dots,x_q\}$, called an {\em extended cluster}, consisting of algebraically independent (over $k$) elements of $\mc{F}$, one for each vertex of $\Delta$.
The elements of $\b{x}$ associated with the mutable vertices are called {\em cluster variables}; they form a {\em cluster}.
The elements associated with the frozen vertices are called
{\em frozen variables}, or {\em coefficient variables}.

A {\em seed mutation} $\mu_u$ at a (mutable) vertex $u$ transforms $(\Delta,\b{x})$ into the seed $(\Delta',\b{x}')=\mu_u(\Delta,\b{x})$ defined as follows.
The new quiver is $\Delta'=\mu_u(\Delta)$.
The new extended cluster is
$\b{x}'=\b{x}\cup\{x_{u}'\}\setminus\{x_u\}$
where the new cluster variable $x_u'$ replacing $x_u$ is determined by the {\em exchange relation}
\begin{equation} \label{eq:exrel}
x_u\,x_u' = \prod_{v\rightarrow u} x_v + \prod_{u\rightarrow w} x_w.
\end{equation}
\end{definition}

\noindent We note that the mutated seed $(\Delta',\b{x}')$ contains the same
coefficient variables as the original seed $(\Delta,\b{x})$.
It is easy to check that one can recover $(\Delta,\b{x})$
from $(\Delta',\b{x}')$ by performing a seed mutation again at $u$.
Two seeds $(\Delta,\b{x})$ and $(\Delta',\b{x}')$ that can be obtained from each other by a sequence of mutations are called {\em mutation-equivalent}, denoted by $(\Delta,\b{x})\sim (\Delta',\b{x}')$.

\begin{definition}[{\em Cluster algebra}] \label{D:cluster-algebra}
The {\em cluster algebra $\mc{C}(\Delta,\b{x})$} associated to a seed $(\Delta,\b{x})$ is defined as the subring of $\mc{F}$
generated by all elements of all extended clusters of the seeds mutation-equivalent to $(\Delta,\b{x})$.
\end{definition}

%Thus, to construct a cluster algebra, one begins with an arbitrary
%{\em initial seed} $(\Delta,\b{x})$ in $\mc{F}$, repeatedly applies
%seed mutations in all possible directions, and takes the $k$-subalgebra
%generated by all elements of $\mc{F}$ appearing in all seeds produced by this
%recursive process.

\noindent Note that the above construction of $\mc{C}(\Delta,\b{x})$ depends only, up to a natural isomorphism, on the mutation equivalence class of the initial quiver $\Delta$. So we may drop $\b{x}$ and simply write $\mc{C}(\Delta)$.

%When a quiver $\Delta$ undergoes a mutation,
%its mutable part (that is, the induced subquiver on the set of mutable
%vertices) does so, too.
%Thus the mutable parts of the quivers at different seeds of a given
%cluster algebra~$\Acal$ are all mutation equivalent.
%This mutation equivalence class determines the (cluster) \emph{type}
%of~$\Acal$.

\subsection{Upper Bounds}

An amazing property of cluster algebras is
\begin{theorem}[{\em Laurent Phenomenon}, \textrm{\cite{FZ1,BFZ}}] \label{T:Laurent}
Any element of a cluster algebra $\mc{C}(\Delta,\b{x})$ can be expressed in terms of the
extended cluster $\b{x}$ as a Laurent polynomial, which is polynomial in coefficient variables.
\end{theorem}

Since $\mc{C}(\Delta,\b{x})$ is generated by cluster variables from the seeds mutation equivalent to $(\Delta,\b{x})$,
Theorem \ref{T:Laurent} can be rephrased as
$$\mc{C}(\Delta,\b{x}) \subseteq \bigcap_{(\Delta',\b{x}') \sim (\Delta,\b{x})}\mc{L}_{\b{x}'},$$
where $\mc{L}_{\b{x}}:=k[x_1^{\pm 1},\dots,x_p^{\pm 1}, x_{p+1}, \dots x_{q}]$.
Note that our definition of $\mc{L}_{\b{x}}$ is slightly different from the original one in \cite{BFZ}, where the coefficient variables are inverted in $\mc{L}_{\b{x}}:=k[x_1^{\pm 1},\dots,x_p^{\pm 1}, x_{p+1}^{\pm 1}, \dots x_{q}^{\pm 1}]$.
\begin{definition}[{\em Upper Cluster Algebra}]
The upper cluster with seed $(\Delta,\b{x})$ is
$$\br{\mc{C}}(\Delta,\b{x}):=\bigcap_{(\Delta',\b{x}') \sim (\Delta,\b{x})}\mc{L}_{\b{x}'}.$$
\end{definition}

In general, there may be infinitely many seeds mutation equivalent to $(\Delta,\b{x})$.
So the above definition is not very useful to test the membership in an upper cluster algebra.
However, the following theorem allows us to do only finitely many checking.

\begin{definition}\cite{BFZ} Let $\b{x}_u (1\leq u \leq p)$ be the {\em adjacent} cluster obtained from $\b{x}$ by applying the mutation at $u$.
We define the upper bounds
$$\mc{U}(\Delta,\b{x}):=\bigcap_{1\leq u\leq p}\mc{L}_{\b{x}_u}.$$
\end{definition}

\begin{theorem}\cite[Corollary 1.9]{BFZ} \label{T:bounds} Suppose that $B(\Delta)$ has full rank, and $(\Delta,\b{x})\sim (\Delta',\b{x}')$,
then $\mc{U}(\Delta,\b{x})=\mc{U}(\Delta',\b{x}')$. In particular, $\mc{U}(\Delta,\b{x})=\br{\mc{C}}(\Delta,\b{x})$.
\end{theorem}

\begin{remark} This theorem is originally proved for $\mc{U}(\Delta,\b{x})$ and $\br{\mc{C}}(\Delta,\b{x})$ with $\mc{L}_{\b{x}}$ defined there.
However, if we carefully examine the argument, we find that the result is also valid for our $\mc{L}_{\b{x}}$.
We only use this theorem in Lemma \ref{L:CCupper}.
If readers are not willing to accept this, please see Remark \ref{R:mu_support}.
\end{remark}

Any (upper) cluster algebra, being a subring of a field, is an integral
domain (and under our conventions, a $k$-algebra).
Conversely, given such a domain~$R$, one may be interested in
identifying $R$ as an (upper) cluster algebra.
As an ambient field~$\mc{F}$,
we can always use the quotient field~$\op{QF}(R)$.
The challenge is to find a seed $(\Delta,\b{x})$ in $\op{QF}(R)$ such
that $\br{\mc{C}}(\Delta,\b{x})=R$. The following lemma is very helpful.

\begin{lemma} \cite[Corollary 3.7]{FP} \label{L:RCA}
Let $R$ be a finitely generated unique factorization domain over $k$.
Let $(\Delta,\b{x})$ be a seed in the quotient field of $R$ such that all elements of $\b{x}$ and all elements of clusters adjacent to $\b{x}$ are irreducible elements of $R$. Then $R \supseteq \br{\mc{C}}(\Delta,\b{x})$.
%If, in addition, the union of extended clusters in $\mc{C}(\Delta,\b{x})$ contains a generating set for $R$, then $R = \mc{C}(\Delta,\b{x})$.
\end{lemma}

\begin{remark} The original conclusion in \cite[Corollary 3.7]{FP} is that $R \supseteq \mc{C}(\Delta,\b{x})$. However, the proof implies this stronger result (see the comment before the proof of \cite[Proposition 3.6]{FP}).
\end{remark}

%\begin{lemma} \label{L:UFDeq} If $\br{\mc{C}}(\Delta)$ is a UFD, then $\br{\mc{C}}(\Delta)=\mc{C}(\Delta)$.
%\end{lemma}

\subsection{Gradings}
\begin{definition} \label{D:wtconfig} A {\em weight configuration} $\bs{\sigma}$ on an ice quiver $\Delta$ is an assignment for each vertex $v$ of $\Delta$ a (weight) vector $\bs{\sigma}(v)$ such that for each mutable $u$, we have that
\begin{equation} \label{eq:weightconfig}
\sum_{v\to u} \bs{\sigma}(v) = \sum_{u\to w} \bs{\sigma}(w).
\end{equation}
A {\em mutation} $\mu_u$ at a mutable vertex $u$ transforms $\bs{\sigma}$ into a weight configuration $\bs{\sigma}'$ of the mutated quiver $\mu_u(\Delta)$ defined as
\begin{equation} \label{eq:mu_wt}
\bs{\sigma}'(v) = \begin{cases} \displaystyle \sum_{u\to w} \bs{\sigma}(w) - \bs{\sigma}(u) & \text{if } v=u, \\ \bs{\sigma}(v) & \text{if } v\neq u. \end{cases}
\end{equation} \end{definition}

\noindent By slight abuse of notation, we can view $\bs{\sigma}$ as a matrix whose $v$-th row is the weight vector $\bs{\sigma}(v)$.
In matrix notation, the condition \eqref{eq:weightconfig} is equivalent to $B\bs{\sigma}$ is a zero matrix, and formula \eqref{eq:mu_wt} is $\bs{\sigma}'=\phi \bs{\sigma}$.
It follows from \eqref{eq:mu_B} and induction that
for any weight configuration of $\Delta$, the mutation can be iterated.

\begin{definition} \label{D:full} A weight configuration $\bs{\sigma}$ is called {\em full} if the null rank of $B^{\T}$ is equal to the rank of $\bs{\sigma}$.
The null space of $B^{\T}$ is called the {\em grading space} of the (upper) cluster algebra $\mc{C}(\Delta)$.
\end{definition}

Given a weight configuration $\bs{\sigma}$ of $\Delta$,
we can assign a multidegree (or weight) to the (upper) cluster algebra $\mc{C}(\Delta,\b{x})$ by setting
$\deg(x_v)=\bs{\sigma}(v)$ for $v=1,2,\dots,q$.
Then mutation preserves multihomogeneousity.
We say that this (upper) cluster algebra is $\bs{\sigma}$-graded, and denoted by $\mc{C}(\Delta,\b{x};\bs{\sigma})$. We refer to $(\Delta,\b{x};\bs{\sigma})$ a graded seed.

\section{Mutations of Quivers with Potentials} \label{S:QP}
In \cite{DWZ1} and \cite{DWZ2}, the mutation of quivers with potentials is invented to model the cluster algebras.
Following \cite{DWZ1}, we define a potential $W$ on a quiver $\Delta$ as a (possibly infinite) linear combination of oriented cycles in $\Delta$.
More precisely, a {\em potential} is an element of the {\em trace space} $\Tr(\ckQ):=\ckQ/[\ckQ,\ckQ]$,
where $\ckQ$ is the completion of the path algebra $k\Delta$ and $[\ckQ,\ckQ]$ is the closure of the commutator subspace of $\ckQ$.
The pair $(\Delta,W)$ is a {\em quiver with potential}, or QP for short.
%For each path $p$ in $\Delta$, the cyclic derivative $\partial_p$ on $\ckQ$ is defined to be the linear extension of the formula
%$$\partial_a W = \sum_{W=uav} vu.$$
For each arrow $a\in \Delta_1$, define the {\em cyclic derivative} of $W$ with respect to $a$ as
$$\partial_a W = \sum_{W=uav} vu.$$
For each potential $W$, its {\em Jacobian ideal} $\partial W$ is the (closed two-sided) ideal in $\ckQ$ generated by all $\partial_a W$.
The {\em Jacobian algebra} $J(\Delta,W)$ is $\widehat{k\Delta}/\partial W$.
A QP is {\em Jacobi-finite} if its Jacobian algebra is finite-dimensional.
If $W$ is polynomial and $J(\Delta,W)$ is finite-dimensional, then the completion is unnecessary to define $J(\Delta,W)$.
This is the case throughout this paper.

At this stage, we assume that each vertex of $\Delta$ is mutable, but later we will freeze some vertices.
The {\em mutation} $\mu_u$ of a QP $(\Delta,W)$ at a vertex $u$ is defined as follows.
The first step is to define the following new QP $\wtd{\mu}_u(\Delta,W)=(\wtd{\Delta},\wtd{W})$.
We put $\wtd{\Delta}_0=\Delta_0$ and $\wtd{\Delta}_1$ is the union of three different kinds
\begin{enumerate}
\item[$\bullet$] all arrows of $\Delta$ not incident to $u$,
\item[$\bullet$] a composite arrow $[ab]$ from $t(a)$ to $h(b)$ for each $a,b$~with~$h(a)=t(b)=u$,
\item[$\bullet$] an opposite arrow $a^*$ (resp. $b^*$) for each incoming arrow $a$ (resp. outgoing arrow $b$) at $u$.
\end{enumerate}
Note that this $\wtd{\Delta}$ is the result of first two steps in Definition \ref{D:Qmu}.
The new potential on $\wtd{\Delta}$ is given by
$$\wtd{W}:=[W]+\sum_{h(a)=t(b)=u}b^*a^*[ab],$$
where $[W]$ is obtained by substituting $[ab]$ for each words $ab$ occurring in $W$. Finally we define $(\Delta',W')=\mu_u(\Delta,W)$ as the {\em reduced part} (\cite[Definition 4.13]{DWZ1}) of $(\wtd{\Delta},\wtd{W})$.
For this last step, we refer readers to \cite[Section 4,5]{DWZ1} for details.

Now we start to define the mutation of decorated representations of $J:=J(\Delta,W)$.
\begin{definition} A decorated representation of the Jacobian algebra $J$ is a pair $\mc{M}=(M,M^+)$,
where $M\in \Rep(J)$, and $M^+$ is a finite-dimensional $k^{\Delta_0}$-module.
By abuse of language, we also say that $\mc{M}$ is a representation of $(\Delta,W)$.
\end{definition}

Consider the resolution of the simple module $S_u$
\begin{align}
\label{eq:Su_proj} \cdots \to \bigoplus_{h(a)=u} P_{t(a)}\xrightarrow{_a(\partial_{[ab]})_b} \bigoplus_{t(b)=u} P_{h(b)} \xrightarrow{_b(b)}  P_u \to S_u\to 0,\\
\label{eq:Su_inj} 0\to S_u \to I_u \xrightarrow{(a)_a} \bigoplus_{h(a)=u} I_{t(a)} \xrightarrow{_a(\partial_{[ab]})_b} \bigoplus_{t(b)=u} I_{h(b)} \to \cdots,
\end{align}
where $I_u$ is the indecomposable injective representation of $J$ corresponding to a vertex $u$.
We thus have the triangle of linear maps with $\beta_u \gamma_u=0$ and $\gamma_u \alpha_u=0$.
$$\vcenter{\xymatrix@C=5ex{
& M(u) \ar[dr]^{\beta_u} \\
\bigoplus_{h(a)=u} M(t(a)) \ar[ur]^{\alpha_u} && \bigoplus_{t(b)=u} M(h(b)) \ar[ll]^{\gamma_u} \\
}}$$

We first define a decorated representation $\wtd{\mc{M}}=(\wtd{M},\wtd{M}^+)$ of $\wtd{\mu}_u(\Delta,W)$.
We set \begin{align*}
&\wtd{M}(v)=M(v),\quad  \wtd{M}^+(v)=M^+(v)\quad (v\neq u); \\
&\wtd{M}(u)=\frac{\Ker \gamma_u}{\Img \beta_u}\oplus \Img \gamma_u \oplus \frac{\Ker \alpha_u}{\Img \gamma_u} \oplus M^+(u),\quad \wtd{M}^+(u)=\frac{\Ker \beta_u}{\Ker \beta_u\cap \Img \alpha_u}.
\end{align*}
We then set $\wtd{M}(a)=M(a)$ for all arrows not incident to $u$, and $\wtd{M}([ab])=M(ab)$.
It is defined in \cite{DWZ1} a choice of linear maps
$\wtd{M}(a^*), \wtd{M}(b^*)$ making
$\wtd{M}$ a representation of $(\wtd{\Delta},\wtd{W})$.
We refer readers to \cite[Section 10]{DWZ2} for details.
Finally, we define $\mc{M}'=\mu_u(\mc{M})$ to be the {\em reduced part} (\cite[Definition 10.4]{DWZ1}) of $\wtd{\mc{M}}$.

Let $\mc{R}ep(J)$ be the set of decorated representations of $J(\Delta,W)$ up to isomorphism. There is a bijection between two additive category $\mc{R}ep(J)$ and $K^2(\proj J)$ mapping any representation $M$ to its minimal presentation in $\Rep(J)$, and the simple representation $S_u^+$ of $k^{\Delta_0}$ to $P_u\to 0$.
Suppose that $\mc{M}$ corresponds to a projective presentation
$P(\beta_1)\to P(\beta_0)$. We see from the exact sequence \eqref{eq:Su_inj} that
\begin{equation} \label{eq:Betti} \beta_1(u)=\dim(\Ker \alpha_u/\Img \gamma_u)+\dim M^+(u), \text{ and } \beta_0(u)=\dim \Coker \alpha_u.\end{equation}

\begin{definition} The {\em $\g$-vector} $\g(\mc{M})$ of a decorated representation $\mc{M}$ is the weight vector of its image in $K^b(\proj J)$, that is, $\g=\beta_1-\beta_0$.
\end{definition}

\begin{remark} Our $\g$-vector is dual to the $\g$-vector considered in \cite{FZ4,DWZ2}. It is the negative of the $\delta$-vector considered in \cite{DF}.
\end{remark}

It follows from \eqref{eq:Betti} that (\cite[(1.13)]{DWZ2})
\begin{equation} \label{eq:delta2dim} \g(u)=\dim(\Coker \gamma_u)-\dim M(u) + \dim M^+(u).
\end{equation}

\begin{definition}\! $\mc{M}$ is called {\em $\g$-coherent} if $\min(\beta_1(u),\beta_0(u))=0$ for all~vertices~$u$,
or equivalently, $\beta_1=[\g]_+$ and $\beta_0=[-\g]_+$. Here, $[\g]_+$ is the vector satisfying $[\g]_+(u) = \max(\g(u),0)$.
\end{definition}

\noindent We also define the mutation $\mu_u$ at the level of $K_0$-groups.
For $\g\in K_0(\proj J)$, $\g':=\mu_u(\g)$ is a vector in $K_0(\proj J')$ defined by ({\em cf.} \cite[(1.3)]{DWZ2})
\begin{equation} \label{eq:mu_delta} \g'(v):=\begin{cases}
-\g(u) & \text{if } u=v;\\
\g(v)-b_{u,v}[-\g(u)]_+ & \text{if } b_{u,v}\geq 0; \\
\g(v)-b_{u,v}[\g(u)]_+ & \text{if } b_{u,v}\leq 0.
\end{cases}\end{equation}
This is also an involution. In general, the $\g$-vectors of $\mc{M}$ and $\mc{M}'$ are related by \eqref{eq:delta_mu_general}.
If we compare \eqref{eq:mu_delta} with \eqref{eq:delta_mu_general} as in the proof of \cite[Theorem 1.7]{DWZ2}, we see that they are equivalent if and only if $\beta_0=[-\g]_+$.

\begin{lemma} \label{L:torsion} The $\g$-vectors of $\mc{M}$ and $\mc{M}'$ are related by~\eqref{eq:mu_delta} if and only if $\mc{M}$ is $\g$-coherent.
%Moreover, being $\g$-coherent is mutation invariant.
\end{lemma}

\begin{definition} \label{D:ice}
An ice quiver with potential, or IQP for short, is a quiver with potential $(\Delta,W)$ with a set $\Delta_\diamond$ of frozen vertices.
We require that there is no arrows between any two frozen verices.
For an IQP $(\Delta,W)$, {\em freezing} a set of vertices $\b{u}$ is the operation that
we set all vertices in $\b{u}$ to be frozen, then delete all arrows between frozen vertices, and set all such arrows to be zero in the potential $W$.
\end{definition}

We freeze some vertices of $\Delta$, and make $(\Delta,W)$ an IQP.
We define the mutation of $(\Delta,W)$ at a mutable vertex $u$ as before except that
we do not create composite arrows $[ab]$ if $t(a)$ and $h(b)$ are both frozen.
%with an additional step in the end called {\em ice reduction}.
%In the ice reduction, we remove all arrows between frozen vertices, and set all such arrows to be zero in the potential.

\begin{lemma} \label{L:commuteFM}
Freezing a set $\b{u}$ of vertices commutes with mutations away from $\b{u}$.
\end{lemma}

\begin{proof} It is easy to see that freezing commutes with the reduction to reduced part.
It suffices to prove the statement for freezing a single vertex $v$ and mutation $\wtd{\mu}_u$ at a single vertex $u$.
It is clear that we get the same quiver if we change the order.
We write the potential $W$ as a sum of two parts $W=W'+W_0$, where the terms of $W_0$ include all cycles containing arrows between $v$ and other frozen vertices.
So if we first freeze $v$, then the potential becomes $W'$.
Then $\wtd{\mu}_u(W')=[W']+\sum b^*a^*[ab]$, where the sum runs through all pairs of arrows $(a,b)$ such that $h(a)=t(b)=u$ and
at least one of $t(a),h(b)$ mutable.
Now we change the order, that is, first apply $\wtd{\mu}_u$.
We have that $\wtd{\mu}_u(W)=[W']+[W_0]+\sum_{} b^*a^*[ab]$, where the summation is the same as before except that $v$ is not frozen.
After freezing $v$, $[W_0]$ still vanishes because $u\neq v$, and clearly two summations agree.
\end{proof}

\noindent So for an IQP $(\Delta,W)$ with $\Delta_\diamond$ frozen, we can equivalently define the mutation at $u$ as first performing the usual mutation at $u$, then freezing $\Delta_\diamond$.

It turns out the right category to look at for an IQP is the category of $\mu$-supported representations in $\mc{R}ep(J)$.
\begin{definition} \label{D:mu_supported}
A decorated representation $\mc{M}$ is called {\em $\mu$-supported} if the supporting vertices of $M$ are all mutable.
We denote by $\mc{R}ep^{\g,\mu}(J)$ the set of all $\g$-coherent $\mu$-supported decorated representations of $J$.
\end{definition}
\noindent  If $\mc{M}$ is not $\mu$-supported, then in general we cannot forget the linear maps between frozen vertices to make it a representation of $J(\wtd{\Delta},\wtd{W})$. But we can do so if $\mc{M}$ is $\mu$-supported.
In this case, $\wtd{\mc{M}}$ is also $\mu$-supported, so
we can define the mutation of $\mc{M}$ as before but without creating linear maps $\wtd{M}([ab])$ if $t(a),h(b)\in \Delta_\diamond$.
We also observe that for $\mu$-supported $\mc{M}$, forgetting the linear maps between frozen vertices does not change the vectors $\beta_1,\beta_0$, and thus the $\g$-vector of $\mc{M}$.

After some mutations, the ice quiver $\Delta$ may acquire some oriented 2-cycles, which would make some mutations undefined for the evolving IQP.
\begin{definition}[\cite{DWZ1}]
We say that a potential $W$ is {\em nondegenerate} for an ice quiver $\Delta$ if any finite sequence of mutations can be applied to $(\Delta,W)$ without creating oriented 2-cycles along the way. Such a IQP $(\Delta,W)$ is also called {\em nondegenerate}.
\end{definition}
\noindent It is known \cite[Proposition 7.3]{DWZ1} that nondegenerate potential always exists for any $2$-acyclic (non-ice) quiver.
For example, one can take a {\em generic} potential of the quiver.
Due to the commuting property of mutation and freezing (Lemma \ref{L:commuteFM}), a nondegenerate potential for a quiver $\Delta$ is also nondegenerate for an ice quiver obtained from $\Delta$ by freezing some vertices.
For a nondegenerate IQP $(\Delta,W)$, the ice quiver in the mutated IQP $\mu_u(\Delta,W)$ is exactly the quiver $\mu_u(\Delta)$ in Definition \ref{D:Qmu}.

\begin{definition}[\cite{DWZ1}] A potential $W$ is called {\em rigid} on a quiver $\Delta$ if
every potential on $\Delta$ is cyclically equivalent to an element in the Jacobian ideal $\partial W$.
Such a QP $(\Delta,W)$ is also called {\em rigid}.
\end{definition}
\noindent It is known \cite[Proposition 8.1, Corollary 6.11]{DWZ1} that every rigid QP is $2$-acyclic, and the rigidity is preserved under mutations. In particular, any rigid QP is nondegenerate.
However, we do not know whether the rigidity is preserved under freezing.

\section{The Cluster Character} \label{S:CC}
Throughout this section, we assume that $(\Delta,W)$ is a nondegenerate IQP.
Let $\b{x}=\{x_1,x_2,\dots,x_q\}$ be an (extended) cluster.
For a vector $\g\in \mb{Z}^q$, we write $\b{x}^\g$ for the monomial $x_1^{\g(1)}x_2^{\g(2)}\cdots x_q^{\g(q)}$.
For $u=1,2,\dots,p$, we set $\hat{y}_u= \b{x}^{-b_{u}}$ where $b_u$ is $u$-th row of the matrix $B(\Delta)$,
and let $\hat{\b{y}}=\{\hat{y}_1,\hat{y}_2,\dots,\hat{y}_p\}$.
The seed mutation of Definition \ref{D:seeds} induces the {\em $\hat{\b{y}}$-seed mutation}.
We recall the mutation rule for $\hat{\b{y}}$ \cite[(3.8)]{FZ4}
\begin{equation} \label{eq:mu_y} \hat{y}_v'=\begin{cases}
\hat{y}_u^{-1} & \text{if } v=u;\\
\hat{y}_v \hat{y}_u^{[b_{v,u}]_+}(\hat{y}_u+1)^{b_{u,v}} & \text{if } v\neq u.
\end{cases}
\end{equation}

\begin{definition}[\cite{DWZ2}] %Let $\mc{M}=(M,M^+)$ be a representation of a QP $(\Delta,W)$.
We define the $F$-polynomial of a representation $M$ by
\begin{equation} F_M(\b{y}) = \sum_\e \chi(\Gr^\e(M)) \b{y}^\e,
\end{equation}
where $\Gr^{\e}(M)$ is the variety parameterizing $\e$-dimensional quotient representations of $M$, and $\chi(-)$ denotes the topological Euler-characteristic.

We also define the {\em cluster character} $C: \mc{R}ep^{\g,\mu}(J) \to \mb{Z}(\b{x})$ by
\begin{equation} \label{eq:CC}
C(\mc{M})=\b{x}^{\g(\mc{M})}F_M(\hat{\b{y}})= \b{x}^{\g(\mc{M})} \sum_{\e} \chi\big(\Gr^{\e}(M) \big) \hat{\b{y}}^{\e}.
\end{equation}
\end{definition}

\noindent Using \eqref{eq:delta2dim}, we can reinterpret \eqref{eq:CC} as ({\em cf.} \cite[Corollary 5.3]{DWZ2})
\begin{equation} \label{eq:CCdim} C(\mc{M})=\b{x}^{-{\sf d}} \sum_{\e} \chi\big( \Gr^{\e}(M) \big) \b{x}^{\epsilon},
\end{equation}
where ${\sf d}=\dv M$ and $\epsilon(u)=\sum_v \big( [b_{v,u}]_+ \e(v) + [b_{u,v}]_+({\sf d}(v)-\e(v)) \big) -\rank \gamma_u +\dim M^+(u).$
The key fact we will use is that the vector $\epsilon$ is non-negative \cite[Corollary 5.5]{DWZ2}.
So the image of $C$ in fact lies in $\mc{L}_{\b{x}}$.

Here is the key Lemma in \cite{DWZ2}.
%It is originally proved for $y_1,y_2,\dots,y_p$ instead of $\hat{y}_1,\hat{y}_2,\dots,\hat{y}_p$,
%where $y_u$ is obtained from $\hat{y}_u$ by evaluating $x_1,\dots,x_p$ at $1$.
%However, exactly the same proof goes through because the mutation rule for $\hat{\b{y}}$ is the same as that for $\b{y}$.

\begin{lemma} \label{L:key} Let $\mc{M}$ be an arbitrary representation of a nondegenerate QP $(\Delta,W)$, and let $\mc{M}'=\mu_u(\mc{M})$, then \begin{enumerate}
\item The $F$-polynomials of ${M}$ and ${M}'$ are related by
\begin{equation} \label{eq:F-poly} (\hat{y}_u+1)^{-\beta_0(u)} F_{M}(\hat{\b{y}})=(\hat{y}_u'+1)^{-\beta_0'(u)} F_{M'}(\hat{\b{y}}').
\end{equation}
\item The $\g$-vector of $\mc{M}$ and $\mc{M}'$ are related by
\begin{equation} \label{eq:delta_mu_general}  \g'(u)=\begin{cases}
-\g(u) & \text{if } u=v;\\
\g(v)+[b_{v,u}]_+\g(u)+b_{v,u}\beta_0(u) & \text{if } u\neq v,
\end{cases}
\end{equation}
and satisfies
\begin{equation} \label{eq:delta_u}  \g(u)=\beta_0'(u) - \beta_0(u).
\end{equation}
\end{enumerate}
\end{lemma}

\begin{lemma} \label{L:CCupper} The mutations commute with the cluster character $C$.
So if $\mc{M}\in\mc{R}ep^{\g,\mu}(J(\Delta,W))$, then $C(\mc{M})$ is an element in the upper cluster algebra $\br{\mc{C}}(\Delta)$.
Moreover, if $(\Delta,\b{x})$ is $\bs{\sigma}$-graded, then
$C(\mc{M})$ is multihomogeneous of degree $\g\bs{\sigma}$.
\end{lemma}

\begin{proof}
The commuting property $\mu_u (C(\mc{M})) = C(\mu_u(\mc{M}))$
is equivalent to that \begin{align*}
\b{x}^\g F_{M}(\hat{\b{y}}) = (\b{x}')^{\g'} F_{M'}(\hat{\b{y}}').
\end{align*}
Comparing with \eqref{eq:F-poly}, we see that it suffices to show that
$$ (\hat{y}_u+1)^{\beta_0(u)} \b{x}^\g = (\hat{y}'_u+1)^{\beta_0'(u)} (\b{x}')^{\g'}.$$
The following equalities are all equivalent to the above.
\begin{align*}
(\hat{y}_u+1)^{\beta_0(u)} \b{x}^\g &= \hat{y}_u^{-\beta_0'(u)}(\hat{y}_u+1)^{\beta_0'(u)} (\b{x}^{[b_u]_+}+\b{x}^{[-b_u]_+})^{\g'(u)} x_u^{-\g'(u)} \b{x}^{\g'-\g'(u)\e_u} & \eqref{eq:exrel},\eqref{eq:mu_y}, \\
\b{x}^{b_u \beta_0'(u)} \b{x}^\g &= (\b{x}^{b_u}+1)^{\g(u)} (\b{x}^{[b_u]_+}+\b{x}^{[-b_u]_+})^{-\g(u)} x_u^{\g(u)} \b{x}^{\g'-\g'(u)\e_u} & \eqref{eq:delta_u},\eqref{eq:mu_delta}, \\
\b{x}^{b_u \beta_1(u)} &= \b{x}^{-[-b_u]_+\g(u)} \b{x}^{[b_u]_+\beta_1(u)-[-b_u]_+\beta_0(u)} & \eqref{eq:delta_u},\eqref{eq:mu_delta}.
\end{align*}
The last equality can be easily verified using $b_u=[b_u]_+-[-b_u]_+$ and $\g=\beta_1-\beta_0$.

To show $C(\mc{M})\in\br{\mc{C}}(\Delta)$, we use a argument similar to that of \cite[Theorem 1.3]{P}.
Let $\b{x}_u$ be the cluster obtained from $\b{x}$ by applying the mutation at $u$.
By Theorem \ref{T:bounds}, it suffices to show that $C(\mc{M})\in\mc{L}_{\b{x}_u}$ for any mutable vertex $u$.
The expression of $C(\mc{M})$ with respect to the cluster $\b{x}_u$ is given by $\mu_{u}(C(\mc{M}))$,
which is $C(\mu_{u}(\mc{M}))\in \mc{L}_{\b{x}_u}$.
%In particular, if $\mc{M}$ can be mutated from the simple $S_u^+$ via a sequence of mutations $\mu_{\b{u}}$, then $C(\mc{M})$ is the cluster variable $\mu_{\b{u}} (x_u)$.
The last statement about the degree follows directly from \eqref{eq:CC} and \eqref{eq:mu_wt}.
In certain sense, this says that the $\g$-vector governs all gradings.
\end{proof}

\begin{remark} \label{R:mu_support} For readers not willing to accept our version of Theorem \ref{T:bounds}, we can argue as follows.
By the version in \cite{BFZ}, it remains to show that $\mu_{\b{u}} (C(\mc{M}))$ is actually polynomial in coefficient variables for any sequence of mutations $\mu_{\b{u}}$.
This is true (initially) for $C(\mc{M})$.
Again by the commuting property $\mu_{\b{u}} (C(\mc{M})) = C(\mu_{\b{u}} (\mc{M}))$,
it is always polynomial in coefficient variables.

Suppose that $\mc{M}$ can be mutated from a positive representation $(0,M^+)$ via a sequence of mutations $\mu_{\b{u}}$.
Since $M^+$ is semisimple, it decomposes as a direct sum of simple representations $\bigoplus_{v\in \Delta_0} {m_v}S_v^+$ with multiplicity $m_v\in \mb{N}_0$.
Then $C(\mc{M})=\mu_{\b{u}}(C(0,M^+))$ is the {cluster monomial} $\prod_{v\in \Delta_0} \mu_{\b{u}}(x_v)^{m_v}$.
By definition, a {\em cluster monomial} is a monomial in elements of any given extended cluster.
\end{remark}

Suppose that an element $z\in\br{\mc{C}}(\Delta)$ can be written as
\begin{equation}\label{eq:z} z = \b{x}^{\g(z)} F(\hat{y}_1,\dots,\hat{y}_p),
\end{equation}
where $F$ is a primitive rational polynomial, and $\g(z)\in \mb{Z}^q$.
If we assume that the matrix $B(\Delta)$ has full rank,
then the elements $\hat{y}_1,\hat{y}_2,\dots,\hat{y}_p$ are algebraically independent so that the vector $\g(z)$ is uniquely determined \cite{FZ4}.
We call the vector~$\g(z)$~the (extended) {\em $\g$-vector} of $z$.
Note that the $\g$-vector of $C(\mc{M})$ as in Lemma \ref{L:CCupper} is $\g(\mc{M})$.
Definition implies at once that for two such elements $z_1,z_2$ we have that
$\g(z_1z_2) = \g(z_1) + \g(z_2)$.
So the set $G(\Delta)$ of all $\g$-vectors in $\br{\mc{C}}(\Delta)$ forms a sub-semigroup of $\mb{Z}^q$.

\begin{lemma} \label{L:independent} Assume that the matrix $B(\Delta)$ has full rank.
Let $Z=\{z_1,z_2,\dots,z_k\}$ be a subset of $\br{\mc{C}}(\Delta)$ with well-defined $\g$-vectors.
If $\g(z_i)$'s are all distinct, then $Z$ is linearly independent over $k$.
\end{lemma}

\begin{proof} A similar statement is proved in \cite{P} in a slightly different setting.
Here, we just adapt his proof to our setting.
The full rank condition on $B$ implies that $\hat{y}_1,\hat{y}_2,\dots,\hat{y}_p$ are algebraically independent,
and moreover we can assign a grading on $x_1,x_2,\dots,x_q$ such that each $\hat{y}_u$ has positive total degree.
Then the total degree of the monomial $\b{x}^{\g}$ is minimal among the total degree of all monomials in \eqref{eq:z}.
Suppose that $\sum_i a_i z_i=0$, then we extract the terms of minimal total degree: $\sum_i a_i \b{x}^{\g(z_i)}=0$.
Since $\g(z_i)$ are all distinct, we conclude that $a_i=0$ for all $i$.
\end{proof}

\noindent It follows from Lemma \ref{L:CCupper} and \ref{L:independent} that
\begin{theorem} \label{T:CC}
Suppose that IQP $(\Delta,W)$ is non-degenerate and $B(\Delta)$ has full rank.
Let $\mc{R}$ be a set of $\g$-coherent $\mu$-supported decorated representations with all distinct $\g$-vectors, then $C$ maps $\mc{R}$ (bijectively) to a set of linearly independent elements in the upper cluster algebra $\br{\mc{C}}(\Delta)$.
\end{theorem}

\begin{definition} \label{D:compatible} If there is some $\mc{R}$ as in Theorem \ref{T:CC} such that
its image under $C$ is a basis of $\br{\mc{C}}(\Delta)$,
then we say that $W$ is a {\em compatible} potential for the upper cluster algebra,
or in short, $W$ is upper-compatible.
\end{definition}

%However, it is not clear to us whether any (upper) cluster algebra has a basis with distinct $\g$-vectors.
The rest of this section is devoted to find a choice of $\mc{R}$ such that $C(\mc{R})$
spans a subspace as large as possible. It turns out that for appropriate $\g$-vectors,
we can take elements of $\mc{R}$ generically in some sense.

%For any $\g\in\mathbb{Z}^{\Delta_0}$, we let $\g_0=[-\g]_+$ and $\g_1=[\g]_+$, then $\g=\g_1-\g_0$.
%If $\mc{M}$ is $\g$-coherent, then $\beta_0=\g_0$.
\begin{definition}
To any $\g\in\mathbb{Z}^{\Delta_0}$ we associate the {\em reduced} presentation space $$\PHom_J(\g):=\Hom_J(P([\g]_+),P([-\g]_+)).$$
We denote by $\Coker(\g)$ the cokernel of a general presentation in $\PHom_J(\g)$.
%We define a map $\psi:K_0(\proj J)\to K_0(\Rep(J))$ by sending $\g$ to $\dim \Coker(\g)$.
\end{definition}
\noindent Reader should be aware that $\Coker(\g)$ is just a notation rather than a specific representation. If we write $M=\Coker(\g)$, this simply means that we take a general presentation in $\PHom_J(\g)$, then let $M$ to be its cokernel.
The following lemma is well-known (see \cite{DF}).
\begin{lemma} \label{L:homotopy} A general presentation in $\Hom_J(P_1,P_0)$ of weight $\g$ is homotopy-equivalent to a general presentation in $\PHom_J(\g)$.
\end{lemma}

\begin{definition}
A weight vector $\g\in K_0(\proj J)$ is called {\em $\mu$-supported} if $\Coker(\g)$ is $\mu$-supported.
Let $G(\Delta,W)$ be the set of all $\mu$-supported vectors in $K_0(\proj J)$.
\end{definition}

A weight vector $\g\in K_0(\proj J)$ is called {\em positive-free} if a general presentation in $\PHom_J(\g)$ does not contain a direct summand of form $P_1\to 0$.
Any $\g\in K_0(\proj J)$ can be decomposed as $\g=\g' + \g^+$ with $\g'$ positive-free and $\g^+\in \mb{Z}_{\geq 0}^q$ such that a general presentation in $\PHom_J(\g)$ is a direct sum of a presentation in $\PHom_J(\g')$ and a presentation in $\PHom_J(\g^+)$.
%By \cite[Corollary 4.2 and Lemma 3.4]{DF}, in this case $\Coker(\g')$ does not support on any support of $\g^+$.

\begin{lemma} \label{L:delta_coh} A general presentation in $\PHom_J(\g)$ corresponds to a $\g$-coherent decorated representation.
%In particular we can take $\mc{R}$ to be the set of decorated representations corresponding to general presentation for each $\g\in G(\Delta,W)$.
\end{lemma}

\begin{proof} We decompose $\g$ as $\g=\g'+\g^+$.
It follows from \cite[Theorem 2.3]{DF} (see also \cite[Corollary 2.8]{P}) that $\Coker(\g')$ is a general representation in some irreducible component $C$, and thus has its minimal presentation in $\PHom_J(\g')$. The result follows.
\end{proof}

\noindent It follows from \cite[Theorem 7.1]{DWZ2} that
\begin{lemma} \label{L:reachable} If $\mc{M}$ can be mutated from a positive representation $(0,M^+)$,
then it corresponds to a {\em rigid} presentation in the sense of \cite{DF}. In particular, the presentation is general.
\end{lemma}

\begin{definition}[\cite{P}]
We define the {\em generic character} $C_W:G(\Delta,W)\to \mb{Z}(\b{x})$~by
\begin{equation} \label{eq:genCC}
C_W(\g)=\b{x}^{\g} \sum_{\e} \chi\big(\Gr^{\e}(\Coker(\g)) \big) \hat{\b{y}}^{\e}.
\end{equation}
\end{definition}

It follows from Theorem \ref{T:CC}, Lemma \ref{L:delta_coh} and \ref{L:reachable} that
\begin{corollary}[{{\em cf.} \cite[Theorem 1.1]{P}}] \label{C:GCC} Suppose that IQP $(\Delta,W)$ is non-degenerate and $B(\Delta)$ has full rank.
The generic character $C_W$ maps $G(\Delta,W)$ (bijectively) to a set of linearly independent elements in $\br{\mc{C}}(\Delta)$ containing all cluster monomials.
\end{corollary}

\noindent A similar result was first proved by P. Plamondon in the setting of the (generalized) cluster category. As mentioned in \cite[Remark 4.1]{P} the result is also valid in the setting of quivers with potentials. The author wants to thank Plamondon for pointing this out to him.
Here, we just gave a simple direct proof. Moreover, the definition of upper cluster algebras in \cite{P} is taken from \cite{BFZ}, so instead of $G(\Delta,W)$ the domain of the generic character is the full lattice $K_0(\proj J)$.

\begin{proposition} \label{P:muG}
$G(\Delta,W)$ is a semigroup, and $G(\mu_u (\Delta,W))=\mu_u (G(\Delta,W))$.
\end{proposition}

\begin{proof} Let $d$ and $d'$ be two general presentations in $\PHom_J(\g)$ and $\PHom_J(\g')$.
If the cokernels of $d$ and $d'$ are not supported on $v$, then so is $d\oplus d'\in \Hom_J(P_1,P_0)$.
By the semi-continuity of the rank function, this is also true for a general presentation in $\Hom_J(P_1,P_0)$.
Note that the weight of $d\oplus d'$ is $\g+\g'$.
Hence, $\g+\g'\in G(\Delta,W)$ by Lemma \ref{L:homotopy}.

Since $\mu_u$ is an involution, it suffices to show $\mu_u (G(\Delta,W)) \subseteq G(\mu_u(\Delta,W))$.
We pick some $\g\in G(\Delta,W)$. Let $\mc{M}$ be some decorated representation corresponding to a general presentation in $\PHom_J(\g)$.
By Lemma \ref{L:delta_coh},  $\mu_u(\g(\mc{M}))=\g(\mu_u(\mc{M}))$.
$\mc{M}$ is $\mu$-supported, then so is $\mu_u(\mc{M})$.
Hence $\mu_u(\g)\in G(\mu_u(\Delta,W))$.
\end{proof}

By Corollary \ref{C:GCC}, we have that $G(\Delta,W)\subseteq G(\Delta)$.
It seems difficult to describe both $G(\Delta,W)$ and $G(\Delta)$ in general.
The following problems are important.

\begin{problem} \label{P:GQ} Under what conditions, does the upper cluster algebra $\br{\mc{C}}(\Delta)$ have a basis such that all its elements have distinct $\g$-vectors?
For which potentials, do we have the equality $G(\Delta)=G(\Delta,W)$?
\end{problem}

\begin{problem} \label{P:saturated}
Is $G(\Delta,W)$ {\em saturated} in $\mb{Z}^q$?
If yes, is $G(\Delta,W)$ consisted of lattice points in some rational polyhedral cone?
We can ask the same question for $G(\Delta)$.
\end{problem}

\noindent We will see in Section \ref{S:Hive} that for ice hive quivers with certain potentials, the answer to both problems are positive.

\section{The Hive Quivers} \label{S:Hive}
The {\em hive quiver $\wtd\Delta_n$ of size $n$} is a quiver in the plane with $\binom {n+2}{2}$ vertices arranged in a triangular grid consisting of $n^2$ small triangles formed by arrows.
We label the vertices as shown in Figure~\ref{fig:hive}.
\begin{figure}[h]
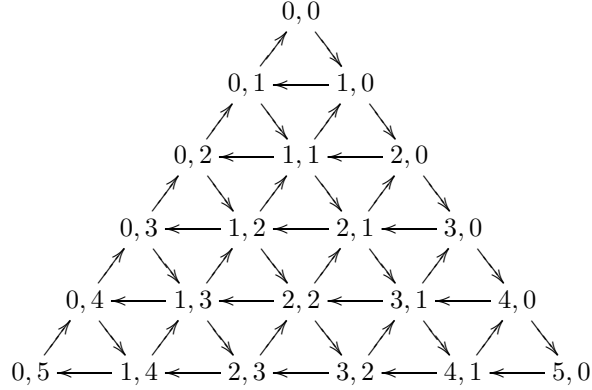

\begin{center}
$\hivefivenoice$
\end{center}
\caption{Hive quiver $\wtd{\Delta}_5$.}
\label{fig:hive}
\end{figure}
The {\em ice hive quiver $\Delta_n$} is obtained from the hive quiver $\wtd\Delta_n$ by forgetting three vertices $\{(0,0),(0,n),(n,0)\}$,
then freezing all boundary vertices as shown in Figure \ref{fig:icehive}.
%We see that the mutable part of $\Delta_n$ is $\wtd{\Delta}_{n-3}$.
We denote by $\delta_n$ the vertex set of $\Delta_n$.
The next two lemmas are straightforward.

\begin{figure}[h]
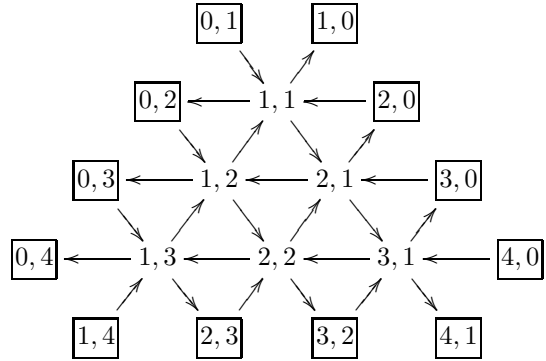

\begin{center}
$\hivefive$
\end{center}
\caption{Ice hive quiver $\Delta_5$.}
\label{fig:icehive}
\end{figure}

\begin{lemma} The $B$-matrix of $\Delta_n$ is of full rank.
\end{lemma}

\noindent Recall the definition of the level-1 weight vector $\f_{i,j}$ in Lemma \ref{L:level1} and \eqref{eq:fij_boundary}.
\begin{lemma} \label{L:WC} The assignment $(i,j)\mapsto \f_{i,j}$ on the vertices of $\Delta_n$
defines a full weight configuration $\bs{\sigma}_n$ on $\Delta_n$.
\end{lemma}

Let $a$ (resp $b$, $c$) denote the sum of all northeast (resp. southeast, west) arrows.
We put the potential $W=abc-acb$ on the quiver $\Delta_n$.
Then the Jacobian ideal is generated by the elements
\begin{align}
\label{eq:rel_Jn1} &e_u(ab-ba), e_u(bc-cb), e_u(ca-ac) & \text{for $u$ mutable,} \\
\label{eq:rel_Jn2} &e_vba, ace_v, e_vac, cbe_v, e_vcb, bae_v & \text{for $v$ frozen.}
\end{align}

\noindent We will identify a path $p$ in $\Delta_n$ with a sequence of vertices.
\begin{definition}  \label{D:Gn} A {\em straight path} in $\Delta_n$ from a vertex $(i,j)$ is any of the following three kinds
\begin{align*}
\tag{NE} &(i,j),(i,j-1),\dots,(i,j-k),\\
\tag{SE} &(i,j),(i+1,j),\dots,(i+k,j),\\
\tag{W} &(i,j),(i-1,j+1),\dots,(i-k,j+k).
\end{align*}
We define a cone $\mr{G}_n\subset \mb{R}^{\delta_n}$ by
$$\begin{cases}
\g(v)\geq 0 & \text{for all frozen vertices $v$,}\\
\sum_{u\in p} \g(u) \geq 0 & \text{for all maximal straight paths $p$ from mutable vertices.}
\end{cases}$$
Here, the notation $\sum_{u\in p} \g(u)$ stands for
$\g(u_1)+\g(u_2)+\cdots+ \g(u_s)$ if $p$ is the path passing $u_1,u_2,\dots,u_s$.
%We denote the lattice points $\mr{G}_n\cap \mb{Z}^{\delta_n}$ by $\mb{G}_n$.
\end{definition}

Let $T_{0,1}$ be the injective (simple) representation $I_{0,1}$.
For $j\geq 2$, let $T_{0,j}$ be the unique indecomposable representation supported on the straight path from $(j,0)$ to $(0,j)$.
Note that $\dv T_{0,j}$ is $(1,1,\dots,1)$ on its support.
\begin{lemma} \label{L:QP_Delta} The IQP $(\Delta_n,W_n)$ is rigid and Jacobi-finite.
Moreover, we have an (injective) presentation of $T_{0,j}$
\begin{align} \label{eq:hive_seq}
& 0\to T_{0,j} \to I_{0,j} \xrightarrow{bc} I_{0,j-1} & \text{for } 2\leq j\leq n-2, \\
& 0\to T_{0,n-1} \to I_{0,n-1} \xrightarrow{(bc,ac)} I_{0,n-2}\oplus I_{1,n-1}. &
\end{align}
\end{lemma}

\begin{proof} The proof of the first statement is similar to that in \cite[Example 8.7]{DWZ1}.
Due to the relation \eqref{eq:rel_Jn1} and \eqref{eq:rel_Jn2},
it is easy to see that for $j\neq n-1$, \begin{enumerate}
\item A dual path in $I_{0,j}$ vanishes under $bc$ if and only if it does not pass $(0,j-1)$;
\item Any dual path to $(0,j)$ not passing $(0,j\!-\!1)$ is equivalent to a straight path.
\item[(1')] A dual path in $I_{0,n-1}$ vanishes under $(bc,ac)$ if and only if it neither passes $(0,n-2)$ nor passes $(1,n-1)$;
\item[(2')] Any dual path to $(0,n-1)$ not passing $(0,n-2)$ or $(1,n-1)$ is equivalent to a straight path.
\end{enumerate}
The exact sequences reformulate the above statements.
\end{proof}

\noindent By symmetry, we also consider the representations $T_{i,0}$ (resp. $T_{i,n-i}$) supported on the straight path from $(i,n-i)$ to $(i,0)$ (resp. from $(0,i)$ to $(i,n-i)$). They have the similar presentations.

\begin{definition} A vertex $v$ is called {\em maximal} in a representation $M$ if all subrepresentations $L\subsetneq M$ are not supported on $v$.
\end{definition}

\noindent Note that each above $T_v$ has a maximal vertex.

\begin{lemma} \label{L:hom=0} Suppose that a representation $T$ contains a maximal vertex $v$. Let $M=\Coker(\g)$, then $\Hom_J(M,T)=0$ if and only if $\g(\dv S)\geq 0$ for all subrepresentations $S$ of $T$.
\end{lemma}

\begin{proof} If $\Hom_J(M,T)=0$, then $\Hom_J(M,S)=0$, and thus $\g(\dv S)\geq 0$ for all subrepresentations $S$ of $T$.
Conversely, suppose that $\PHom_J(\g)=\Hom_J(P_1,P_0)$.
We add $c=\g(\dv T)$ copies of $P_v$'s to $P_0$ so that a general presentation $P_1\to P_0\oplus cP_v \to M' \to 0$ has weight $\g'=\g+ c \e_v$.
It satisfies that $\g'(\dv T)=0$ and $\g'(\dv S)=\g(\dv S)\geq 0$ for all subrepresentations $S \subsetneq T$.
By King's criterion (Lemma \ref{L:King}), we see that $T$ is $\g'$-semi-stable, and thus $\Hom_J(M',T)=0$.
Now a general presentation $P_1\xrightarrow{(f,g)} P_0\oplus cP_v$ must have $f$ general in $\Hom_J(P_1,P_0)$.
Hence, $\Hom_J(M',T)=0$ implies $\Hom_J(M,T)=0$.
\end{proof}

\begin{lemma} \label{L:TIequi} Let $M=\Coker(\g)$, then $\Hom_J(M,T_v)=0$ for each frozen $v$ if and only if $\Hom_J(M,I_v)=0$ for each frozen $v$.
\end{lemma}

\begin{proof}  Since each subrepresentation of $T_v$ is also a subrepresentation of $I_v$, one direction is clear.
Conversely, let us assume that $\Hom_J(M,T_v)=0$ for each frozen $v$. We prove that $\Hom_J(M,I_{0,j})=0$ by induction on $k$.
For $k=1$, we have that $T_{0,1}=I_{0,1}$.
Now suppose that it is true for $k=j-1$, that is, $\Hom_J(M,T_{0,j-1})=0$.
By Lemma \ref{L:QP_Delta}, for $j<n-1$ $\Hom_J(M,I_{0,j})=0$ is equivalent to $\Hom_J(M,T_{0,j})=0$.
$\Hom_J(M,I_{0,n-1})=0$ is equivalent to $\Hom_J(M,T_{0,n-1})=0$ and $\Hom_J(M,T_{1,n-1})=0$.
We are done by symmetry and induction.
\end{proof}

\begin{theorem} \label{T:LP_Gn} The set of lattice points $\mr{G}_n\cap \mb{Z}^{\delta_n}$ is exactly $G(\Delta_n,W_n)$.
\end{theorem}

\begin{proof} Due to Lemma \ref{L:hom=0} and \ref{L:TIequi}, it suffices to show that $\mr{G}_n$ is defined by $\g(\dv S)\geq 0$ for all subrepresentations $S$ of $T_v$ and all $T_v$.
We notice that these defining conditions are the union of the defining conditions of $\mr{G}_n$ and $\g(\dv T_v)\geq 0$.
But the latter conditions for $v\neq (0,1),(n-1,0),(1,n-1)$ are clearly redundant.
\end{proof}

\begin{remark} For any IQP $(\Delta,W)$, we can defined a cone $\mr{G}_{I}(\Delta,W)\subset \mb{R}^{\Delta_0}$ by $\g(\dv S)\geq 0$ for all subrepresentations $S\subseteq I_v$ and all frozen vertices $v$.
It is not hard to see that the cone $\mr{G}_n = \mr{G}_{I}(\Delta_n,W_n)$.
It is also clear that $G(\Delta,W)$ is always contained in $\mr{G}_{I}(\Delta,W) \cap \mb{Z}^{\Delta_0}$. We are curious about the next problem.
\end{remark}

\begin{problem} For what kind of IQP, we have the equality
$$G(\Delta,W)=\mr{G}_{I}(\Delta,W) \cap \mb{Z}^{\Delta_0}?$$
\end{problem}

For a fixed weight vector $\sigma$, we consider the convex polytope $\mr{G}_n(\sigma)$ obtained from $\mr{G}_n$ by adding the condition
$\g \bs{\sigma}_n =\sigma$. We will show in next section that there is a volume-preserving linear transformation mapping $\mr{G}_n(\sigma)$ onto a hive polytope of Knutson-Tao.

\section{Littlewood-Richardson Triangles} \label{S:LR}

%We consider the underlying graph of the hive quiver $\Delta_n$.
%Let $H_n$ denote the vector space of all labelings
%$h=(h_{ij})$ of the vertices of $\Delta_n$ with real numbers.
%The dimension of $H_n$ is clearly $\binom{n+2}{2}-3$.
\begin{definition}[\cite{PVa}]
A {\em Littlewood-Richardson triangle of size $n$} is an element $\h=\big(\h(i,j)\big)\in \mb{R}^{\delta_n}$ that satisfies the following conditions:
\begin{align}
\label{eq:Tineq1} & \h(i,j)\ge 0, & & \text{for  $ij\neq 0$,} \\
\label{eq:Tineq2} & \sum_{k=0}^{i-j} \h(k,j) \ge \sum_{k=0}^{i-j} \h(k,j+1),
& & \text{for  $1\le j\le i \leq n-1$},\\
\label{eq:Tineq3} &\sum_{k=0}^{j-1} \h(i-k,k) \ge \sum_{k=0}^{j} \h(i+1-k,k), &
 & \text{for  $1\le j\le i \leq n-1$}.
\end{align}
\end{definition}

\noindent
%Note that the inequality
%\begin{equation}
%\sum_{p=0}^{j} h_{pj} \ge \sum_{p=0}^{j+1} h_{p\, j+1},
%\text{ for $1\le j < n$.}
%\label{redundante}
%\end{equation}
%follows from (CS) with $i=j$  and (LR) with $i=j$.
We denote by $\lr n$ the cone of all Littlewood-Richardson triangles in $\mb{R}^{\delta_n}$. %and call it a {\em LR cone}.
To each $\h=\big(\h(i,j)\big) \in \lr n$ we associate the following numbers:
\begin{align*}
&\mu_i = \h(i,0), & & \text{for $1\le i\le n-1$,}\\
&\nu_j = \sum_{k=1}^{n-j} \h(k,j), & & \text{for $1\le j\le n-1$,}\\
&\lambda_i = \sum_{k=0}^i \h(i-k,k), & & \text{for $1\le i\le n$.}
\end{align*}
Then it follows from \eqref{eq:Tineq1}--\eqref{eq:Tineq3} that $\lambda,\mu$ and $\nu$ are partitions with $|\lambda| = |\mu| + |\nu|$.
We call $\tripar$ the {\em type} of $\h$, and denote by  $\lr n\tripar$
the set of all LR triangles of type $\tripar$;
this is a convex polytope.
It is proved in \cite{PVa} that there is a volume-preserving linear transformation mapping $\lr n \tripar$ onto a hive polytope of Knutson-Tao \cite{KT}. In particular,

\begin{lemma} \cite[Corollary 4.2]{PVa} \label{L:LP_LR}
$\lr n \tripar$ has $\lrcoef$ integral points.
\end{lemma}

For any weight $\sigma$, we get a triple of partition $(\lambda(\sigma),\mu(\sigma),\nu(\sigma))$ via \eqref{eq:wt2par1}--\eqref{eq:wt2par3}.
We denote the polytope $\lr n(\lambda(\sigma),\mu(\sigma),\nu(\sigma))$ by $\lr n (\sigma)$.
It is clear that $\lr n (1^{i+j},1^i,1^j)$ has a unique integral point $\h_{i,j}$ satisfying
$$\begin{cases} \h_{i,j}(k,0)=\h_{i,j}(i,l)=1  &  \text{if } 1\leq k \leq i, 1\leq l \leq j, \\
\h_{ij}(k,l)=0 & \text{otherwise}.
\end{cases}$$

%It is clear that $\lr n (1^{i+j},1^i,1^j)$ has a unique integral point $\h_{i,j}$ as shown in the picture, where unfilled spots are all zeros.
%$$\intij$$

\begin{theorem} \label{T:VPiso} There is a volume preserving linear isomorphism $\mb{R}^{\delta_n} \to \mb{R}^{\delta_n}$ mapping the polytope $\mr{G}_n(\sigma)$ onto the polytope $\lr n (\sigma)$. In particular, $\mr{G}_n(\sigma)$ has $c_{\mu(\sigma),\nu(\sigma)}^{\lambda(\sigma)}$ integral points.
\end{theorem}

\begin{proof} The definition of the isomorphism is obvious.
Let $\e_{i,j}$ be the standard basis of $\mb{R}^{\delta_n}$, and $\f_{i,j}$ be the level-1 weight vector defined in Section \ref{S:Tn}.
By Lemma \ref{L:level1_par} and the remark above, there is a unique integral point $\h_{i,j}\in \lr n(\f_{i,j})\subset \mb{R}^{\delta_n}$.
Then the assignment $\e_{i,j}\mapsto \h_{i,j}$ induces a linear map $\varphi:\mb{R}^{\delta_n}\to \mb{R}^{\delta_n}$.
We need to show that \begin{enumerate}
\item $\varphi$ is volume preserving.
\item $\varphi$ pulls the supporting hyperplanes of $\lr n(\sigma)$ back to those of $\mr{G}_n(\sigma)$.
\end{enumerate}
We order the standard basis of $\mb{R}^{\delta_n}$ according to the lexicographic order of the subindices, that is,
${\sf E} = \{\e_{0,1},\e_{0,2},\dots,\e_{1,0},\e_{1,1},\dots,\e_{n-1,0},\e_{n-1,1} \}$.
The matrix of $\varphi$ with respect to ${\sf E}$ is upper triangular with ones on the main diagonal.
So it has determinant one, and thus volume preserving.

We write the definition of $\varphi$ in coordinates:
$$\varphi(\g)=\h, \quad \text{where } \h(i,j)=
\begin{cases} \sum_{l\geq j} \g(i,l) & \text{if $j\neq 0$}, \\ \sum_{\substack{k\geq j}, l}  \g(k,l) & \text{if $j=0$}.
\end{cases}$$
It is easy to verify case by case that
\eqref{eq:Tineq1} matches with $\sum_{(i,j)\in p} \g(i,j)\geq 0$ for the straight paths $p$ of type (NE) and $\g(i,j)\geq 0$ for $j=0$;
\eqref{eq:Tineq2} matches with the type (SE) and $\g(i,j)\geq 0$ for $i+j=n$;
\eqref{eq:Tineq3} matches with the type (W) and $\g(i,j)\geq 0$ for $i=0$.
Finally, by our construction the type of $\h$ is $\g\bs{\sigma}_n=\sigma$.
\end{proof}

\section{Cluster Structure in $\SI_\bn(T_n)$} \label{S:CS}
\subsection{Algebraic Independence}
In this subsection, we show that the set of level-1 semi-invariants
$$\mc{S}_n:=\{s(f_{ij})\in \SI_\bn(T_{n})\mid (i,j)\in \delta_{n}\}$$
are algebraically independent. The proof is a little technical, so we suggest that readers skip this part for the first-time reading.

We first recall a {\em slice theorem} in \cite{S1}.
Let $Q$ be any finite quiver without oriented cycles.
For a representation $M$ of $Q$,
the right {\em orthogonal category} $M^\perp$ is the abelian subcategory $\{M\in\Rep(Q)\mid M\perp N\}$.
A representation $E$ is called {\em exceptional} if $\Hom_Q(E,E)=k$ and $\Ext_Q(E,E)=0$, so the dimension vector of $E$ corresponds to a real Schur root $\epsilon$.

Schofield showed that the category $E^\perp$ is equivalent to the category $\Rep(Q_E)$ for another quiver $Q_E$ with one vertex less than $Q$. We compose this equivalence with the embedding $E^\perp \hookrightarrow\Rep(Q)$, and obtain a functor $\iota_E:\Rep(Q_E)\hookrightarrow\Rep(Q)$.
It induces a linear inclusion of $K_0(\Rep(Q_E))$ into $K_0(\Rep(Q))$.
If $E\perp \beta$, which means that $E$ is right orthogonal to a general representation in $\Rep_\beta(Q)$,
then we denote by $\beta_\epsilon$ the dimension vector of $Q_E$ such that $\beta$ is $\beta_\epsilon$ under the inclusion.
%Let $\Rep_\beta(E^\perp)$ be the space of all $\beta$-dimensional representations left orthogonal to $E$.

\begin{theorem} \label{T:homofibre}\cite[Theorem 3.2]{S1} If $E\perp \beta$, then $\Rep_\beta(E^\perp):=\Rep_\beta(Q)\cap E^\perp$ is isomorphic to the homogeneous fibre space  $\GL_\beta\times_{\GL_{\beta_\epsilon}}\Rep_{\beta_\epsilon}(Q_E)$.
\end{theorem}

\noindent It follows (\cite[Corollary 6.11]{Fm}, see also \cite{DW2}) that if $E\perp \beta$, then we have an isomorphism
$\SI_{\beta_\epsilon}^{^\sigma(\alpha_\epsilon)}(Q_E)\cong \SI_\beta^{^\sigma\!\alpha}(Q).$
So we have an embedding $\iota: \SI_{\beta_\epsilon}(Q_E)\hookrightarrow \SI_{\beta}(Q)$.
%Concretely, for $M\in \Rep(Q_E)$, $\iota$ sends $s(M)$ to $s(\iota_E(M))$.

%Schofield constructed two embeddings
%$$\Phi: \Rep_{\beta_\epsilon}(Q_E) \hookrightarrow \Rep_\beta(E^\perp)\quad \text{and}\quad \GL_{\beta_\epsilon} \hookrightarrow \GL_\beta,$$
%such that $\Rep_\beta(E^\perp)$ is $\GL_{\beta_\epsilon}$-stable.

\begin{lemma} \label{L:localSI}
Let $s:=s(E)$ and $\SI_\beta(Q)_{s}$ be the localization of $\SI_\beta(Q)$ at $s$, then there is an embedding mapping $x$ to $s$
$$\SI_{\beta_\epsilon}(Q_E)[x,x^{-1}] \hookrightarrow \SI_\beta(Q)_{s}.$$
In particular, if $\{s_1,\dots,s_n\}\subset \SI_{\beta_\epsilon}(Q_E)$ is algebraically independent, then so is
$\{s,\iota(s_1),\dots,\iota(s_n)\}$.
%Assume in addition that $^\sigma\!\epsilon$ is extremal in $\Sigma_\beta(Q)$, then the embedding is an isomorphism.
\end{lemma}

\begin{proof} We define an algebra morphism
$$\varphi: \SI_{\beta_\epsilon}(Q_E)[x,x^{-1}]\to \SI_\beta(Q)_{s}\quad \text{by}\quad
rx^d\mapsto \iota(r)s^d.$$
To show $\varphi$ is injective, we suppose that $\varphi(\sum_d r_dx^d)=\sum_d \iota(r_d)s^d=0$.
We define a partial order on the weights such that $\sigma_1\geq \sigma_2$ if and only if $\sigma_1-\sigma_2=n ({^\sigma\!\epsilon})$ for some $n\in \mb{N}_0$.
Now the leading term of $\sum_d \iota(r_d)s^d$ has the highest weight because $\epsilon \notin \epsilon^\perp:=\{\alpha\mid \innerprod{\epsilon,\alpha}_Q=0\}$ and the weight of $\iota(r_d)$ is in $^\sigma(\epsilon^\perp)$.
So we see inductively that each $\iota(r_d)$, and thus each $r_d$, has to vanish.

%We do not need the last statement in this paper, but the proof is not hard.
%We assume that $^\sigma\!\epsilon$ is extremal in $\Sigma_\beta(Q)$.
%To show $\varphi$ is surjective, by Theorem \ref{T:inv_span} it suffices to show that
%for any $s(M)s^{-d}\in \SI_\beta(Q)_{s}$, we can always rewrite it as $s(M)s^c s^{-d-c}$ such that $s(M)s^c\in \Img \iota$.
%Since $^\sigma\!\epsilon$ is extremal, the universal homomorphism $hE\to M$ must be injective, where $h=\hom_Q(E,M)$.
%Otherwise, the kernel $K$ gives a non-zeros semi-invariant $s(K)$, which contradicts $^\sigma\!\epsilon$ being extremal.
%Let $C$ be the cokernel of $hE\hookrightarrow M$,
%then we consider the universal extension $0\to C\to \tilde{M} \to eE\to 0$, where $e=\ext_Q(E,M)$.
%By construction, we have that $\tilde{M}\in E^\perp$.
%By Lemma \ref{L:exact}, up to a scalar $s(\tilde{M})=s(M) s^c \in \Img\iota$, where $c=e-h=-\innerprod{\epsilon,\alpha}_Q$.
\end{proof}

\begin{remark} \label{R:exseq} The above construction together with Lemma \ref{L:localSI} can be inductively generalized from an exceptional representation to an exceptional sequence $\mb{E}$.
We refer the readers to \cite[Section 2]{DW2} for details.
\end{remark}
We recall that an {\em exceptional sequence of dimension vector} $\mb{E}:=\{e_1,e_2,\dots, e_n\}$ is a sequence of real Schur roots of $Q$ such that $e_i\perp e_j$ for any $i<j$.
It is called {\em quiver} if $\innerprod{e_j,e_i}_Q\leq 0$ for any $i<j$.
It is called {\em complete} if $n=|Q_0|$.
The {\em quiver} of a quiver exceptional sequence $\mb{E}$ is by definition the quiver with vertices labeled by $e_i$ and $\ext_Q(e_j,e_i)$ arrows from $e_j$ to $e_i$.
According to \cite[Theorem 4.1]{S2}, $\ext_Q(e_j,e_i)=-\innerprod{e_j,e_i}_Q$.
Now we return to the triple flag quivers.

\begin{lemma} \label{L:exsequence}
$$\mb{E}:=\{ {^\alpha(\f_{0,n-1})},{^\alpha(\f_{n-1,0})},{^\alpha(\e_1^3)},\e_{n-1}^1+\e_{n-1}^2+\e_n,\e_{n-1}^3,\dots,\e_2^3,\e_{n-2}^2,\dots,\e_1^2,\e_{n-2}^1,\dots,\e_1^1\}$$
is a complete quiver exceptional sequence in $\Rep(T_n)$
such that the quiver $T_{n-1}'$ of
$$\{{^\alpha(\e_1^3)},\e_{n-1}^1+\e_{n-1}^2+\e_n,\e_{n-1}^3,\dots,\e_2^3,\e_{n-2}^2,\dots,\e_1^2,\e_{n-2}^1,\dots,\e_1^1\}$$
is
$$\vcenter{\xymatrix@R=3ex{
\e_1^1\ar[r] & \e_2^1 \ar[r] & \cdots \ar[r]  & \e_{n-2}^1 \ar[dr] \\
\e_1^2\ar[r] & \e_2^2 \ar[r] & \cdots \ar[r]  & \e_{n-2}^2 \ar[r] & \e_{n-1}^1+\e_{n-1}^2+\e_n \ar[r]^{\qquad a} & ^\alpha(\e_1^3)\\
\e_2^3\ar[r] & \e_3^3 \ar[r] & \cdots \ar[r]  & \e_{n-1}^3 \ar[ur]
}}$$
\end{lemma}

\begin{proof} It is trivial to check that $\mb{E}$ is a complete quiver exceptional sequence following the definition.
Alternatively, one can perform the {\em mutation} operators \cite{R} on the standard sequence
$$\mb{E}_0:=\{\e_n,\e_{n-1}^3,\dots,\e_1^3,\e_{n-1}^2,\dots,\e_1^2,\e_{n-1}^1,\dots,\e_1^1\}$$
as follows.
\begin{enumerate}
\item Apply the left mutations to move $\e_1^3$ to the left, and it becomes ${^\alpha(\e_1^3)}$;
\item Apply the left mutations to move $\e_{n-1}^2$ and $\e_{n-1}^1$ right next to $\e_n$, and they remain themselves;
\item Apply the right mutations 2 times to move $\e_n$ passing $\e_{n-1}^2$ and $\e_{n-1}^1$, and it becomes $\e_{n-1}^1+\e_{n-1}^2+\e_n$;
\item Apply the left mutations 2 times to move $\e_{n-1}^2$ and $\e_{n-1}^1$ passing ${^\alpha(\e_1^3)}$, and they become ${^\alpha(\f_{0,n-1})}$ and ${^\alpha(\f_{n-1,0})}$.
\end{enumerate}
\end{proof}

\noindent To simplify the notation, we label the rightmost vertex of $T_{n-1}'$ by $n$, so $\e_n$ is the unit vector supported on $n$. We also observe that $T_{n-1}'$ contains $T_{n-1}$ as a subquiver. So the weight vector $\f_{i,j}$ for $T_{n-1}$ is naturally a weight vector for $T_{n-1}'$ if we set its $n$-th coordinate to be zero.

Let $\mb{F}$ be the exceptional sequence $\{ {^\alpha(\f_{0,n-1})},{^\alpha(\f_{n-1,0}}) \}$. By Lemma \ref{L:exsequence} and definition, the quiver $T_{n-1}'$ is the quiver $(T_n)_{\mb{F}}$.
So we have an embedding $\Rep((T_n)_{\mb{F}})\cong \Rep(\mb{F}^\perp)\hookrightarrow \Rep(T_n)$.
Let $\beta_{n-1}'$ be the following dimension vector of $T_{n-1}'$.
$$\vcenter{\xymatrix@R=3ex{
1\ar[r] & 2 \ar[r] & \cdots \ar[r]  & n-2 \ar[dr] \\
1\ar[r] & 2 \ar[r] & \cdots \ar[r]  & n-2 \ar[r] & n-1 \ar[r] & 1\\
1\ar[r] & 2 \ar[r] & \cdots \ar[r]  & n-2 \ar[ur]
}}$$

\begin{lemma} \label{L:embed_Tn-1}
Under the embedding, the dimension vector $\beta_{n-1}'$ goes to the standard one $\beta_n$ of $T_n$.
Moreover, $\f_{i,j}$ goes to
\begin{align*}
&\f_{i,j} & \text{for } & i+j=n-1, ij\neq 0,\\
&\f_{i,j}+\f_{n-1,0} &  \text{for } & i=0,\\
&\f_{i,j}+\f_{0,n-1} & \text{for } & j=0,\\
&\f_{i,j}+\f_{0,n-1}+\f_{n-1,0} & \text{for } & 1< i+j< n-1,ij\neq 0;
\intertext{$\f_i:=\e_n+\e_{n-1}-\e_i^1-\e_{n-i}^2$ goes to }
&\f_{i,n-i} & \text{for } & 1\leq i \leq n-1.
\end{align*}
Here, we use the convention that $\e_{n-1}^1=\e_{n-1}^2=\e_{n-1}$.
Moreover, let $L,N_1,N_2$ be the cokernel of $f_{i,j},f_{n-1,0},f_{0,n-1}$ respectively,
and $M_1,M_2,M$ be general representations of dimension ${^\alpha(\f_{i,j}+\f_{n-1,0})},{^\alpha(\f_{i,j}+\f_{0,n-1})},{^\alpha(\f_{i,j}+\f_{0,n-1}+\f_{n-1,0})}$ respectively.
Then we have exact sequences \begin{align*}
&0\to L\to M_1 \to N_1\to 0 &  \text{for } & i=0, \\
&0\to L\to M_2 \to N_2\to 0 &  \text{for } & j=0, \\
&0\to M_1\to M \to N_2\to 0 &  \text{for } & 1< i+j< n-1,ij\neq 0.
\end{align*}
\end{lemma}

\begin{proof} Let us recall the recipe of the linear inclusion $\iota_\epsilon$ of the Grothendieck groups.
Let $\beta$ be a dimension vector of $T_{n-1}'$, then $\iota_\epsilon(\beta)=\sum_v \beta(v) e_v$, where $e_v$ is the real Schur root in the exceptional sequence corresponding to the vertex $v$.
The recipe for a weight vector $\f$ of $T_{n-1}'$ is similar: $\iota_\epsilon(\f)=\sum_v {(^\alpha\!\f)}(v) {^\sigma(e_v)}$.
Then the first part can be easily verified.

By Lemma \ref{L:level1}, $L,N_1$ and $N_2$ are general representations.
It is easy to check that $\hom_{T_n}(L,N_i)=\hom_{T_n}(N_i,L)=0$,
$\innerprod{L,N_i}_{T_n}=0$, and $\innerprod{N_i,L}_{T_n}=-1$,
so $\ext_{T_n}(L,N_i)=0$ and $\ext_{T_n}(N_i,L)=1$.
According to Lemma \ref{L:ext}, we have the first two exact sequences.
The existence of the last exact sequence is proved similarly.
%We apply $\Hom_{T_n}(-,N_2)$ and $\Hom_{T_n}(N_2,-)$ to the first exact sequence:
%\begin{align*}
%0&=\Ext_{T_n}(N_1,N_2)\to \Ext_{T_n}(M_1,N_2)\to \Ext_{T_n}(L,N_2)=0,\\
%0&=\Hom_{T_n}(N_2,N_1)\to \Ext_{T_n}(N_2,L)\to \Ext_{T_n}(N_2,M_1)\to \Ext_{T_n}(N_2,N_1)=0.
%\end{align*}
%We conclude that $\ext_{T_n}(M_1,N_2)=0$ and $\ext_{T_n}(N_2,M_1)=1$.
\end{proof}

\begin{lemma} \label{L:2arm} A general representation in $\Rep_\bn(T_n)$ has a following representative. Its matrices on the first arm are $(I_k\ 0)$; on the second arm are $(0\ I_k)$, where $I_k$ is a $k\times k$ identity matrix for $k=1\dots,n-1$ and $0$ is a zero column vector.
\end{lemma}

\begin{proof} If we forget the third arm (but keep the central vertex $n$), then we get a quiver of type $A_{2n-1}$ with the dimension vector shown below
$$\vcenter{\xymatrix@C=5ex{
1\ar[r] & 2\ar[r] & \cdots \ar[r] & n-1\ar[r] & n & n-1\ar[l] & \cdots \ar[l] & 2 \ar[l] & 1 \ar[l]
}}$$
This dimension vector decomposes canonically as $\sum_{i=0}^{n-1} \bold{1}_i$, where
$$\bold{1}_i=(\underbrace{0,\cdots, 0}_i, \underbrace{1,\cdots,1}_n,0,\cdots,0).$$
This means that a general representation of that dimension vector is a direct sum of general representations of dimension $\bold{1}_i$,
which is exactly what we desired.
\end{proof}

Recall the weight vector $\f_i=\e_n+\e_{n-1}-\e_i^1-\e_{n-i}^2$ in Lemma \ref{L:embed_Tn-1}. We associate for each $\f_i$ a projective presentation $f_i$
$$P_{n-1}\oplus P_n \xrightarrow{\sm{p_i^1 & p_{n-i}^2 \\ p_i^1a & 0}} P_i^1\oplus P_{n-i}^2 $$
and thus a semi-invariant function $s(f_i)\in\SI_{\beta_{n-1}'}(T_{n-1}')$. Here $a$ is the rightmost arrow as shown in the figure of Lemma \ref{L:exsequence}.
From now on, we will write $s_i$ and $s_{i,j}$ for $s(f_i)$ and $s(f_{i,j})$.

%\begin{lemma} Consider the product decomposition $\Rep_{\beta'}(T_{n-1}')=\Rep_{\beta}(T_{n-1})\times V$.
%Let $R$ be semi-invariant ring $\SI_\beta(T_{n-1})$ embedded in .
%The semi-invariants $\{s(f_i)\}$
%are algebraically independent over $R$.
%\end{lemma}

\begin{lemma} \label{L:induction}
Assume that elements in $\mc{S}_{n-1}$ are algebraically independent over $k$.
Then $\{s_i\mid 1\leq i\leq n-1\}\cup \mc{S}_{n-1}$ are algebraically independent as well.
\end{lemma}

\begin{proof} We consider a representation $M\in \Rep_{\beta_{n-1}'}(T_{n-1}')$ whose matrices on the first arm are $(I_k \ 0)$; on the second arm are $(0 \ I_k)$; on the third arm are in general position; and the matrix for the arrow $a$ is generic $(x_1,x_2,\dots,x_{n-1})^{\T}$.
Then the semi-invariant $s_i$ is the determinant of the block matrix
$$\begin{pmatrix} I_i & O_{i,n-i-1} & X_i \\ O_{n-i,i-1}& I_{n-i} & O_{n-i,1} \end{pmatrix}$$
where $I_i$ is the $i\times i$ identity matrix, $O_{i,j}$ is the $i\times j$ zero matrix, and $X_i=(x_{1},\dots,x_{i})^{\T}$.
It is easy to see that this determinant equals to $(-1)^{n-i} x_i$.

Now suppose that $\{s_i\mid 1\leq i\leq n-1\}\cup \mc{S}_{n-1}$ are algebraically dependent, that is, there is a non-zero polynomial $p\in k[y_1,y_2,\dots,y_{\frac{1}{2}(n^2+3n-8)}]$ such that
$$p\big( s_1,\dots,s_{n-1}, \{ s_{i,j} \}_{(i,j)\in \delta_{n-1}} \big)=0.$$
Its total degree on $y_1,\dots,y_{n-1}$ must be strictly positive because elements in $\mc{S}_{n-1}$ are algebraically independent.
We evaluate $p$ at the representation $M$, and get
$$p\big( (-1)^{n-1}x_1,\dots,-x_{n-1},\{s_{i,j}(M)\}_{(i,j)\in \delta_{n-1}} \big)=0.$$
Since $x_1,\dots,x_{n-1}$ are generic variables, we conclude that there is some function $r\in R$ vanishes at $M$, where $R\subset \SI_{\beta_{n-1}}(T_{n-1})$ is the subring generated by $\mc{S}_{n-1}$.
By Lemma \ref{L:2arm}, the restriction of $M$ on the subquiver $T_{n-1}$ is a general representation of dimension $\beta_{n-1}$.
But a general representation in $\Rep_{\beta_{n-1}}(T_{n-1})$ is stable, so
there is no such $r$. We get a contradiction.
\end{proof}

%Before proving the theorem, let us first recall the fact \cite{E} that for a finitely generated integral domain $R$ over $k$, the Krull dimension of $R$ is equal to the transcendence degree of the field of fractions of $R$ over $k$.

\begin{theorem} \label{T:ag-independent} The semi-invariant functions in $\mc{S}_n$ are algebraically independent over the base field $k$.
\end{theorem}

\begin{proof} We prove by induction on $n$. If $n=2$, then the statement is trivial.
Now suppose that it is true for $n=m$.
Applying Lemma \ref{L:localSI} (see Remark \ref{R:exseq}) to the exceptional sequence $\mb{F}=\{{^\alpha(\f_{0,m})},{^\alpha(\f_{m,0})}\}$, we get
$$\SI_{\beta_m'}(T_m')[x^{\pm 1},y^{\pm 1}] \hookrightarrow \SI_{\beta_{m+1}}(T_{m+1})_{s_{0,m}s_{m,0}},$$
where the quiver $T_m'$ and dimension vector $\beta'$ are as in Lemma \ref{L:embed_Tn-1}.

By induction and Lemma \ref{L:induction}, we have that
$$\{s_{i,j}\mid (i,j)\in \delta_m\} \cup \{s_i\mid 1\leq i \leq m\}$$ are algebraically independent.
Then by Lemma \ref{L:localSI}, \ref{L:embed_Tn-1}, and \ref{L:exact}
\begin{align*} & \{s_{m,0},s_{0,m}\} \cup
\{s_{0,j}s_{m,0}, s_{i,0}s_{0,m} \mid 1\leq i,j \leq m \} \cup\\
& \{s_{i,j}s_{m,0}s_{0,m} \mid {1<i+j<m, ij\neq 0} \} \cup
\{s_{i,j} \mid {i+j=m, m+1, ij\neq 0} \}
\end{align*}
is algebraically independent.
We notice that all elements above, up to some multiple of $s_{m,0}$ and $s_{0,m}$, are elements in $\{s_{i,j} \mid (i,j)\in \delta_{m+1} \}$.
This implies that level-1 semi-invariants in $T_{m+1}$ are algebraically independent.
We finish the proof by induction.
\end{proof}

\subsection{Initial Exchanges}
Recall the weight configuration $\bs{\sigma}_n$ for $\Delta_n$.
Let $\g_{i,j}= \f_{i-1,j}+\f_{i,j+1}+\f_{i+1,j-1}= \f_{i+1,j}+\f_{i,j-1}+\f_{i-1,j+1}$.
It is not hard to verify that $^\alpha(\g_{i,j})$ is an {\em isotropic} Schur root of $T_n$, but we do not need this fact.
If we mutate $(\Delta_n,\bs{\sigma}_n)$ at a mutable vertex $(i,j)$, then
$$\f_{i,j}' = \g_{i,j} - \f_{i,j} = 2\e_n-\e_{i-1}^1-\e_{i+1}^1-\e_{j-1}^2-\e_{j+1}^2-\e_{k-1}^3-\e_{k+1}^3,\qquad i,j,k\geq 1,i+j+k=n.$$
\noindent As before, we use the convention that $\e_{0}^a$ is the zero vector for $a=1,2,3$.
%They correspond to mutation at a vertex node, an edge node, and an interior node respectively.
%\begin{align*}
%& 2\e_n-\e_2^a-\e_2^b-\e_{n-3}^c-\e_{n-1}^c, & & \\
%& 2\e_n-\e_{2}^a+\e_{j-1}^b-\e_{j+1}^b-\e_{k-1}^c-\e_{k+1}^c,  & &  j,k>1, j+k=n-1,\\
%& 2\e_n-\e_{i-1}^1-\e_{i+1}^1-\e_{j-1}^2-\e_{j+1}^2-\e_{k-1}^3-\e_{k+1}^3, & & i,j,k>1,i+j+k=n.
%\end{align*}

\begin{lemma} \label{L:f'} For each $(i,j)\in \delta_n$, we have that
\begin{enumerate}
% The dimension vector $^\alpha(\g_{i,j})$ is an isotropic Schur root of $T_n$
\item $\dim \SI_\bn^{\g_{i,j}}(T_n)=2$ and $\dim \SI_\bn^{\f_{i,j}'}(T_n)=1$.
\item The weight $\f_{i,j}'$ is extremal in $\Sigma_\bn(T_n)$, so $\SI_\bn^{\f_{i,j}'}(T_n)$ is spanned by an irreducible polynomial.
%In particular, $s(f')$ is irreducible.
\end{enumerate} \end{lemma}

\begin{proof} (1). From the correspondence \eqref{eq:wt2par1}--\eqref{eq:wt2par3}, we need to show that
$c_{\mu(\g_{i,j})\nu(\g_{i,j})}^{\lambda(\g_{i,j})}=2$.
It is easy to see that the transposed partitions $\lambda(\g_{i,j})^*,\mu(\g_{i,j})^*,\nu(\g_{i,j})^*$ is given by
$$(3,2,1)+(i+j-2)(1,1,1), (2,1,0)+(i-1)(1,1,1), (2,1,0)+(j-1)(1,1,1).$$
%$$(n,n-1,n-2)-k(1,1,1), (m,m-1,m-2)-k(1,1,1), (n-m+1,n-m,n-m-1).$$
So everything reduces to show $\lrcoef = 2$ for $\tripar=\big((3,2,1),(2,1),(2,1)\big)$.
This is clear from the Littlewood-Richards rule.
Similarly, we show that $c_{\mu(\f_{i,j}')\nu(\f_{i,j}')}^{\lambda(\f_{i,j}')}=1$.
The transposed partitions $\lambda(\f_{i,j}')^*,\mu(\f_{i,j}')^*,\nu(\f_{i,j}')^*$ is given by
$$(3,1)+(i+j-2)(1,1), (2,0)+(i-1)(1,1), (2,0)+(j-1)(1,1).$$
So we only need to observe that $c_{(2),(2)}^{(3,1)}=1$.

%Being level-2 is clear from the definition.
(2). Suppose that $\f_{i,j}' = \f_1 + \f_2$ with $\f_1,\f_2\in \Sigma_\bn(T_n)$,
then $\f_1$ and $\f_2$ must be two level-1 weights.
But this is clearly impossible. So $\f_{i,j}'$ spans an extremal ray.
According to Lemma \ref{L:irreducible}, $\SI_\bn^{\f_{i,j}'}(T_n)$ is spanned by an irreducible polynomial,
\end{proof}

For each $\f_{i,j}'$, we associate a Schofield's semi-invariant function corresponding to the presentation
\begin{align*}
%& P_2^b \oplus P_{n-3}^c \oplus P_2^a\oplus P_{n-1}^c \xrightarrow{\sm{p_2^b & p_{n-3}^c & 0 & p_{n-1}^c\\0 & 0 & p_2^a & p_{n-1}^c}^{\T}} 2P_n, \\ %f^{c}
%& P_{j+1}^b\oplus P_{k-1}^c\oplus P_{2}^a\oplus P_{j-1}^b\oplus P_{k+1}^c \xrightarrow{\sm{p_{j+1}^b & p_{k-1}^c & 0 & 0 & p_{k+1}^c \\0 & 0 & p_{2}^a & p_{j-1}^b & p_{k+1}^c}^{\T}} 2P_n, \\ %f_{j}^{b,c}
& P_{i-1}^1\oplus P_{j+1}^2\oplus P_{k-1}^3\oplus P_{i+1}^1\oplus P_{j-1}^2\oplus P_{k+1}^3\xrightarrow{\sm{p_{i-1}^1 & p_{j+1}^2 & p_{k-1}^3 & 0 & 0 & p_{k+1}^3 \\0 & 0 & 0 & p_{i+1}^1 & p_{j-1}^2  & p_{k+1}^3}^{\T}} 2P_n.
\end{align*}

\noindent We keep using the short hand $s_{i,j}$ and $s_{i,j}'$ for Schofield's semi-invariants $s(f_{i,j})$ and $s(f_{i,j}')$.
The relations below are the initial exchange relations needed to apply Lemma \ref{L:RCA}.
\begin{lemma} \label{L:exchange} We have the following relations
\begin{align*}
%& (-1)^n s_{1,1}s_{1,1}' = s_{0,1}s_{1,2}s_{2,0}+s_{2,1}s_{1,0}s_{0,2},\\
%& (-1)^n s_{1,j}s_{1,j}' = s_{0,j}s_{1,j+1}s_{2,j-1}+s_{2,j}s_{1,j-1}s_{0,j+1},\\
& (-1)^n s_{i,j}s_{i,j}' = s_{i-1,j} s_{i,j+1} s_{i+1,j-1} + s_{i+1,j} s_{i,j-1} s_{i-1,j+1}.
\end{align*}
%For the first two, there are similar relations for $s_{j,1}s_{j,1}'$ and $s_{n-1-j,j}s_{n-1-j,j}'$ by symmetry.
\end{lemma}

\begin{proof}
%We only prove the statement for $(i,j)$ an interior vertex, i.e., $f_{i,j}'$ takes the third form.
%The other two is similar but simpler, so we leave them for readers.
Let \begin{align*}
F_0=(-1)^n s_{i,j}s_{i,j}',\ F_1= s_{i-1,j}s_{i,j+1}s_{i+1,j-1},\ F_2= s_{i+1,j}s_{i,j-1}s_{i-1,j+1}.
\end{align*}
Since $\dim \SI_\bn^{\g_{i,j}}(T_n)=2$, we must have that
$a_0 F_0 = a_1 F_1  + a_2 F_2,$ for some $a_0,a_1,a_2\in k$.
To find $a_i$, we consider a special representation $M$,
whose matrices on the first arm are $(I_k \ 0)$;
on the second arm are $(0 \ I_k)$;
on the third arm are $(I_k \ 0)$ for $k=1,\dots,n-2$,
where $I_k$ and $0$ are as in Lemma \ref{L:2arm}.
The last matrix on the third arm is
$$M(a_{n-1}^3)_{rc}=\begin{cases} 1 & \text{ if } r-c=i-1,i,i+1 \\ 0 & \text{ otherwise.} \end{cases}.$$
%It is easy to check that
%$$\prod_{k=k_0}^{k_1} (I_k\ 0) = (I_{k_0}\ 0),\quad \text{ and }\quad \prod_{k=k_0}^{k_1} (0\ I_k) = (0\ I_{k_0}).$$
%Here, the $0$ on the right side is an appropriate zero block.
Specializing the semi-invariants at $M$, we find that
$s_{i-1,j},s_{i+1,j-1},s_{i+1,j},s_{i-1,j+1}$ are triangular with diagonal all ones.
Moreover, elementary linear algebra calculation shows that $s_{i,j}'=(-1)^{n}$ and
$$\big( s_{i,j},s_{i,j+1},s_{i,j-1} \big)=
\begin{cases} (-1)^q(1,1,0) & j=3q \\ (-1)^q(1,0,1) & j=3q+1 \\ (-1)^q(0,-1,1) & j=3q+2. \end{cases}$$
Similarly, we consider another representation $M'$ which is the same as $M$ except that
the last matrix $M(a_{n-1}^2)$ on the second arm is replaced by $(I_{n-1} \ 0)+(0\ I_{n-1})$.
We find that the specialization of $s_{i-1,j},s_{i+1,j-1},s_{i+1,j},s_{i-1,j+1},s_{i,j}'$ at $M'$ is the same as at $M$.
But $$\big( s_{i,j},s_{i,j+1},s_{i,j-1}\big)=
\begin{cases} (-1)^q(1,0,1) & j=3q \\ (-1)^q(0,-1,1) & j=3q+1 \\ (-1)^q(-1,-1,0) & j=3q+2. \end{cases}$$
Therefore, we conclude that $a_0=a_1=a_2$ by solving a linear system.
\end{proof}

\subsection{Main Theorem}
\begin{theorem} \label{T:equal} We assign for each vertex $(i,j)$ of the hive quiver $\Delta_n$ the semi-invariant function
$s(f_{ij})$.
Then $\SI_\bn(T_n)$ is the upper cluster algebra $\br{\mc{C}}(\Delta_n,\mc{S}_n)$.
\end{theorem}

\begin{proof}  Due to Theorem \ref{T:ag-independent}, such an assignment does define a cluster algebra $\mc{C}(\Delta_n,\mc{S}_n)$.
By \cite[Theorem 3.17]{PV}, any $\SI_\bn(T_n)$ is a UFD.
Due to Lemma \ref{L:exchange} and \ref{L:f'}.(2), we can apply Lemma \ref{L:RCA} and get the containment $\br{\mc{C}}(\Delta_n,\mc{S}_n)\subseteq \SI_\bn(T_n)$.
By Corollary \ref{C:GCC}, $C_{W_n}(\mr{G}_n(\sigma)\cap \mb{Z}^{\delta_n})$ is a linearly independent set in $\SI_\bn(T_n)_\sigma$ for any weight $\sigma$.
But the dimension of $\SI_\beta(T_n)_\sigma$ is counted by $|\mr{G}_n(\sigma)\cap \mb{Z}^{\delta_n}|$ according to Theorem \ref{T:eqLR} and \ref{T:VPiso}.
Hence, we have the equality $\br{\mc{C}}(\Delta_n,\mc{S}_n)= \SI_\bn(T_n)$.
\end{proof}

\noindent It follows from the proof that
\begin{corollary} \label{C:basis} The generic character $C_{W_n}$ maps the lattice points in $\mr{G}_n$ (bijectively) onto a basis of $\br{\mc{C}}(\Delta_n)$.
\end{corollary}
\noindent This can be viewed as an algebraic analogue of the full Fock-Goncharnov conjecture for $\SL_n$ \cite{FG}.

\begin{proposition} \label{P:finite} Suppose that a cluster algebra $\mc{C}(\Delta)$ has only finitely many clusters.
If all extremal rays in the $\g$-vector cone $\mb{R}_+G(\Delta)$ are generated by the $\g$-vectors of cluster variables,
then $G(\Delta)$ is generated by the $\g$-vectors of cluster and coefficient variables over $\mb{Z}_+$.
\end{proposition}

\begin{proof} Let $\b{x}=(x_1,x_2,\dots,x_q)$ be a cluster with $\g$-vectors $\b{g}=(\g_1,\g_2,\dots, \g_q)$.
%In particular, $\mb{Z}_+\b{g} = \mb{R}_+ \b{g} \cap \mb{Z}^q$.
Let $\mr{G_r}$ be the union of $\mb{R}_+ \b{g}$ for all clusters.
%It forms a normal fan.
It is proved in \cite[Proposition 6.1]{DF} that for coefficient-free cluster algebras, the finiteness of clusters implies that $\mr{G_r}$ covers the whole $\mb{R}^q$.
The extended $\g$-vector of a cluster variable is uniquely determined by its principal part (see the remark after \cite[Definition 3.7]{T}).
So the same argument as in \cite{DF} can show that in general $\mr{G_r}$ is a polyhedral cone, and thus equal to $\mb{R}_+G(\Delta)$.
It is known \cite[Theorem 1.7]{DWZ2} that $\b{g}$ form a $\mb{Z}$-basis of the lattice $\mb{Z}^q$.
Since $\mr{G_r}$ is a union of unimodular cones $\mb{R}_+ \b{g}$, $G(\Delta)$ is generated by the $\g$-vectors of cluster and coefficient variables over $\mb{Z}_+$.
\end{proof}

\begin{corollary} For $n<6$, we have that $\br{\mc{C}}(\Delta_n)=\mc{C}(\Delta_n)$.
\end{corollary}

\begin{proof} The first nontrivial case is $n=3$. This case is obvious.
For $n=4,5$, we use the software MPT \cite{MPT} to compute all extremal rays of $\mr{G}_n$.
We find that there are 18 extremal rays and 45 extremal rays in $\mr{G}_4$ and $\mr{G}_5$.
Checking with \cite{Ke}, they are all generated by the $\g$-vectors of cluster and coefficient variables.
Our claim then follows from Proposition \ref{P:finite} and Corollary \ref{C:basis}.
\end{proof}

\noindent We can show that $\br{\mc{C}}(\Delta_6)=\mc{C}(\Delta_6)$ but the proof is more involved \cite{Fs2}.
We do not know if we always have the equality $\br{\mc{C}}(\Delta_n)=\mc{C}(\Delta_n)$.

\section{Constructing Generating Sets} \label{S:smalln}

\subsection{Minimal Cluster Generating Set}

Let $Q$ be a finite quiver without oriented cycles as in Section \ref{S:SI}.
Since $\SL_\beta$ is reductive, we know from the Hochster-Roberts theorem that
the semi-invariant ring $\SI_\beta(Q)$ is finitely generated $k$-algebra.
Now we assume that $\SI_\beta(Q)$ is a cluster algebra naturally graded by weights of semi-invariants.
Although the cluster algebra may contain infinitely many cluster variables,
we can always choose a finite minimal set $F$ of generators consisting of cluster variables and coefficient variables.
We call such a set a {\em minimal cluster generating set} of $\SI_\beta(Q)$.
Recall that all cluster variables and coefficient variables are all multihomogeneous.

Let $F'$ be another minimal set of (not necessary multihomogeneous) generators of $\SI_\beta(Q)$.
We claim that $F$ and $F'$ have the same cardinality.
We know that $\SI_\beta(Q)$ is {\em positively multigraded} $k$-algebra in the sense that the degree zero component $\SI_\beta(Q)_{\b{0}}$ is $k$.
We can assume that all elements in $F'$ have no constant term,
then they belong to the ideal $J$ of $\SI_\beta(Q)$ generated by all multihomogeneous elements of nonzero degree.
By the (positively multigraded version of) Nakayama lemma, $F'$ modulo $J^2$ forms a basis of $J/J^2$.
The same argument can be applied to $F$ as well.

The idea to construct a minimal set of multihomogeneous generators is quite naive.
We compute for each weight $\sigma\in \Sigma_\beta(Q)$ the subspace $R_\sigma$ of $\SI_\beta(Q)_\sigma$
$$R_\sigma:=\sum_{\sigma'\in \Sigma_\beta(Q)} \SI_\beta(Q)_{\sigma'} \SI_\beta(Q)_{\sigma-\sigma'}.$$
Let $B_\sigma$ be a basis of any subspace of $\SI_\beta(Q)_\sigma$ complementary to $R_\sigma$.
Then $\bigcup_\sigma B_\sigma$ constitutes a minimal set of multihomogeneous generators for $\SI_\beta(Q)$.
All such minimal sets can be so obtained.
In particular, we can choose each element in $B_\sigma$ to be a cluster variable or a coefficient variable.

\begin{conjecture} \label{P:cluster_gen} If $\SI_\beta(Q)$ is a cluster algebra naturally graded by weights of semi-invariants,
then it has a unique minimal cluster generating set.
\end{conjecture}

%\begin{proof} Suppose that there is another such set $\bigcup_\sigma B_\sigma'$,
%then in particular there is some $\sigma$ such that $B_\sigma\neq B_\sigma'$.
%Pick some $f\in B_\sigma' \setminus B_\sigma$, then we can express $f$ as a linear combination of some elements $R_\sigma$ and $B_\sigma$.
%\end{proof}

\subsection{$\cone{n}$ Revisited}
Now we come back to the complete triple flags.
We will construct a minimal cluster generating set for $\SI_\bn(T_n)$ when $n< 6$.
To carry out our algorithm, we should start with all indivisible extremal weights $\sigma_{\op{e}}$ in $\Sigma_\bn(T_n)$.
Note that for any such weight $\sigma_{\op{e}}$, $\SI_\beta(Q)_{\sigma_{\op{e}}}$ is spanned by cluster variables and coefficient variables over $k$.

Recall that we have Algorithm CT to find extremal rays of $\cone{n}$.
Now with Theorem \ref{T:equal}, we can improve the algorithm as follows.
We first compute all extremal rays $\g_{\op{e}}$ in $\mr{G}_n$,
then all extremal rays in $\cone{n}$ are among $\g_{\op{e}} \bs{\sigma}_n$.

\begin{example}
In contrast to \cite[Conjecture 0.19]{GHKK}, not all extremal rays of $\mr{G}_n$ are $\g$-vectors of cluster variables.
The first counterexample appears when $n=6$. Let us consider an extremal ray of $\mr{G}_6$
$$\g_0=(\e_{2,0}+\e_{0,4}+\e_{4,2})+\e_{2,3}+\e_{3,1}+\e_{1,2}-\e_{2,1}-\e_{1,3}-\e_{3,2}.$$
We say a $\g$-vector $\g$ can be obtained via a sequence of mutations $\mu_{\b{u}}:=\mu_{u_k}\cdots \mu_{u_2} \mu_{u_1}$ from $\Delta$
if $\g$ is the $\g$-vector of the cluster variable $\mu_{\b{u}}(x_{u_k})$ (with respect to the seed $(\Delta,\b{x})$).
It is known from \cite{DWZ2,DF} that if $\g_0$ can be reached via mutations, then the {\em $\op{E}$-space} $\op{E}(f,f)=0$ for $f=\Coker(\g_0)$.
But it is not hard to show that $\op{E}(f,f)=k$.
\end{example}

The following proposition can greatly expedite our search for minimal cluster generating sets.
\begin{proposition}[{\cite[Corollary 7.12]{DW2}}] For $n\leq 7$,
the extremal rays of $\cone{n}$ are generated by weights corresponding to real Schur roots.
\end{proposition}

As pointed in \cite[Example 7.13]{DW2}, $n=8$ is the first example when $\cone{n}$ can contain an extremal weight $\sigma$ such that
$^\alpha\! \sigma$ is an {isotropic} Schur root.
%In general, the extremal rays of $\cone{n}$ can contain imaginary Schur roots.
The example given there is $$\sigma=3\e_8-(\e_1^1+\e_3^2+\e_3^3+\e_5^1+\e_6^2+\e_6^3).$$
In this case, $c_{\mu(\sigma),\nu(\sigma)}^{\lambda(\sigma)}=2$.
It turns out that $\SI_{\beta_8}(T_8)_{\sigma}$ is spanned by two cluster variables.
Interested readers can verify using \cite{Ke} that the two cluster variables have $\g$-vectors
\begin{align*}
& (\e_{0,6}+\e_{5,3}+\e_{2,0})+\e_{1,3}+\e_{3,2}-\e_{2,2}-\e_{3,3}, \\
& (\e_{0,3}+\e_{5,3}+\e_{2,0})+\e_{4,1}+\e_{1,6}-\e_{4,3}-\e_{2,1},
\intertext{which can be obtained respectively by the sequences of mutations from $\Delta_8$}
&(3,1),(2,2),(2,1),(3,1),(4,1),(3,2),(4,2),(5,2),(4,3),(3,4),(3,3),(2,4),(1,5),(2,5);\\
&(2,5),(2,4),(3,4),(3,3),(4,3),(4,2),(4,1),(3,2),(3,1),(2,2),(2,1),(1,2).
\end{align*}

\subsection{When $n\leq 6$} {\ }
The first nontrivial case appears when $n=3$. It~is~treated~in~\cite{SW}.
\begin{proposition} The semi-invariant ring $\SI_{\beta_3}(T_3)$ is a cluster algebra of type $A_1$. It can be presented by
$$k[s_{0,1},s_{1,0},s_{0,2},s_{2,0},s_{1,2},s_{2,1},s_{1,1},s_{1,1}']/(s_{1,1}s_{1,1}'-s_{0,1}s_{2,0}s_{1,2}-s_{1,0}s_{0,2}s_{2,1}).$$
\end{proposition}

Below we will use a well-known fact that the number of cluster variables in a cluster algebra $\mc{\Delta}$ is completely determined by the mutable part of $\Delta$. It is equal to the number of real roots of $\Delta$ plus $|\Delta_0|$ if $\Delta$ is acyclic.

\begin{proposition} The semi-invariant ring $\SI_{\beta_4}(T_4)$ is a cluster algebra of type $A_3$. %It can be presented by
It is minimally generated by 9 cluster variables and 9 coefficient variables in $\mc{C}(\Delta_4,\mc{S}_4)$.
They are all of extremal weights in $\Sigma_{\beta_4}(T_4)$.
\end{proposition}

\begin{proof}
It is enough to observe that there are 9 cluster variables for a cluster algebra of type $A_3$.
Together with the 9 coefficient variables, each one has to match one of the 18 extremal rays in the cone $\cone{4}$ (see Example \ref{Ex:T4}).
\end{proof}

\begin{proposition} The semi-invariant ring $\SI_{\beta_5}(T_5)$ is a cluster algebra of type $D_6$.
It is minimally generated by 30 out of 36 cluster variables and 12 coefficients of $\mc{C}(\Delta_5,\mc{S}_5)$.
They are all of extremal weights in $\Sigma_{\beta_5}(T_5)$.
The weights of six cluster variables excluded are, up to symmetry, equal to
$\sigma=3\e_5-(\e_1^3+\e_2^1+\e_2^2+\e_3^2+\e_3^3+\e_4^1).$
%The $\g$-vectors of six excluded are (up to symmetry)
%\begin{align*}
%&\g_1:=\e_{1,0}+\e_{0,3}+\e_{3,2}+\e_{2,1}-\e_{1,2},\\
%%\e_{2,0}+\e_{0,4}+\e_{3,2}+\e_{1,2}-\e_{2,2},\\
%%\e_{2,0}+\e_{0,3}+\e_{4,1}+\e_{2,2}-\e_{2,1};\\
%&\g_2:=\e_{1,0}+\e_{0,2}+\e_{3,2}+\e_{2,1}+\e_{1,3}-\e_{1,1}-\e_{2,2}.
%\end{align*}
\end{proposition}

\begin{proof} If we perform a sequence of mutations at (1,3),(2,1),(1,1), and (1,2).
then the mutable part of $\Delta_5$ transforms into a Dynkin quiver of type $D_6$, which has 30 real roots.
We use Algorithm CT to find that there are 42 extremal rays in $\cone{5}$.
They correspond to the 30 cluster variables and 12 coefficients.
The weights of the remaining 6 cluster variables are, up to symmetry, equal to $\sigma$.
We need to show these 6 cluster variables are redundant.
The weight $\sigma$ corresponds to the partition $\lambda=(5,4,2),\mu=(4,2),\nu=(3,2)$.
By Theorem \ref{T:eqLR}, the dimension of $\SI_{\beta_5}(T_5)_\sigma$ is equal to $\lrcoef$, which is equal to 2 by the Littlewood-Richardson rule.
Consider the products of cluster variables
$$C_{W_5}(\e_{1,0}+\e_{0,3}+\e_{3,2}-\e_{1,2})x_{2,1} \text{ \ and \ } x_{2,3}x_{3,0}x_{0,1}.$$
The cluster variable with $\g$-vector $\e_{1,0}+\e_{0,3}+\e_{3,2}-\e_{1,2}$ can be obtained by a sequence of mutations at
$\{(1,1),(2,2),(2,1),(1,2)\}$.
The both products have weight $\sigma$ but different $\g$-vectors, so they are linearly independent.
\end{proof}

%\begin{conjecture} The cluster algebra $\br{\mc{C}}(\Delta_n,W_n)$ is generated by $C_{W_n}(\g)$ for all indivisible $\g$ extremal in $G(\Delta_n,W_n)$.
%\end{conjecture}
%

In general, $\SI_{\bn}(T_n)$ is not generated by semi-invariants of extremal weights.
The following proposition will be proved in \cite{Fs2} using the technique of projections.
\begin{proposition} The semi-invariant ring $\SI_{\beta_6}(T_6)$ is a cluster algebra of infinite type.
It is minimally generated by 103 cluster variables, 15 coefficients of $\mc{C}(\Delta_6,\mc{S}_6)$.
Six of these cluster variables are not of extremal weights.
Their weights up to symmetry are equal to
$\sigma=3\e_6-(\e_1^3+\e_2^1+\e_2^2+\e_4^2+\e_4^3+\e_5^1).$
\end{proposition}

\noindent Here, at least we can see that the six cluster variables of non-extremal weights are necessary.
We use Algorithm CT to find that there are 112 extremal rays in $\mb{R}^+\Sigma_{\beta_6}(T_6)$.
They corresponds to the 97 cluster variables and 15 coefficients.
The weight $\sigma$ corresponds to the partition $\lambda=(6,5,2),\mu=(5,2),\nu=(4,2)$ with $\lrcoef=2$.
However, we can only find one decomposition of $\sigma$ into extremal weights,
$$\sigma=(\e_6-\e_1^3-\e_5^1)+(\e_6-\e_2^1-\e_4^2)+(\e_6-\e_2^2-\e_4^3).$$

%\section{QP}
%The {\em mutation} of quivers with potentials and their decorated representation is defined in \cite{DWZ1}.
%This definition can be easily extended to projective presentations as follows.
%\begin{align*}
%&\mu_u(P_w\xrightarrow{f} P_v) = P_w'\xrightarrow{\mu_u(f)} P_v', \\
%&\mu_u(P_w\xrightarrow{f} P_u)=P_w'\oplus P_u' \xrightarrow{\binom{\wtd{\mu_u(f)}}{a}} \bigoplus_{v\to u} P_v', \\
%&\mu_u(P_u\xrightarrow{f} P_v)=\bigoplus_{u\to w} P_w' \xrightarrow{(b\ \wtd{\mu_u(f)})} P_u'\oplus P_v'.
%\end{align*}
%
%\begin{lemma} The mutation of presentations is compatible with the mutation of decorated representations.
%That is, we have the commutative diagram
%\end{lemma}

\section*{Acknowledgement}
This project began to be conceived during the MCR conference on Cluster Algebras in Snowbird June 2014. The author would like to thank Professor Gordana Todorov for invitation, Professor Harm Derksen for support, and the organizers for providing ideal working environment.
He also want to thank Professor Jerzy Weyman for carefully reading the manuscript and giving several good suggestions.

\bibliographystyle{amsplain}

\begin{thebibliography}{19}
%\bibitem {ASS}  I. Assem, D. Simson, A. Skowro\'{n}ski \textit{Elements of the Representation Theory of Associative Algebras,} London Mathematical Society Student Texts 65, Cambridge University Press, 2006.
\bibitem {BFZ} A. Berenstein, S. Fomin, A. Zelevinsky, \textit{Cluster algebras. III. Upper bounds and double Bruhat cells,} Duke Math. J. 126 (2005), no. 1, 1--52.
\bibitem {C} A-M. Castravet, \textit{The Cox ring of $\br{M}_{0,6}$,} Trans. Amer. Math. Soc. 361 (2009), 3851--3878.
\bibitem {DF} H. Derksen, J. Fei, \textit{General presentations of algebras,} to appear Adv. Math.
\bibitem {DW1} H. Derksen, J. Weyman, \textit{Semi-invariants of quivers and saturation for Littlewood-Richardson coefficients,} J. Amer. Math. Soc. 13 (2000), no. 3, 467--479.
%\bibitem {DWc}  H. Derksen, J. Weyman, \textit{On the canonical decomposition of quiver representations,} Compositio Math. 133 (2002), no. 3, 245--265.
\bibitem {DW2} H. Derksen, J. Weyman, \textit{The combinatorics of quiver representation,}  Ann. Inst. Fourier (Grenoble) 61 (2011), no. 3, 1061--1131.
\bibitem {DWZ1}  H. Derksen, J. Weyman, A. Zelevinsky, \textit{Quivers with potentials and their representations I,} Selecta Math. (N.S.) 14 (2008), no. 1, 59--119.
\bibitem {DWZ2}  H. Derksen, J. Weyman, A. Zelevinsky, \textit{Quivers with potentials and their representations II,} J. Amer. Math. Soc. 23 (2010), no. 3, 749--790.
%\bibitem {F}  J. Fei \textit{General presentations of algebras,} Ph.D. thesis University of Michigan, 2010, available online:
%{http://deepblue.lib.umich.edu/bitstream/2027.42/77740/1/jiarui\_1.pdf}
\bibitem {Fm} J. Fei, \textit{Moduli of representations I. Projections from quivers,} arXiv:1011.6106.
\bibitem {Fs2} J. Fei, \textit{Cluster algebras and semi-invariant rings II. Projections and embeddings,} In preparation.
\bibitem {Fp} J. Fei, \textit{Presenting cluster algebras,} In preparation.
\bibitem {Ft4} J. Fei, \textit{A geometric study of $\SI_\beta(T_4)$,} Unpublished manuscript 2014.
\bibitem {FG} V. Fock and A. Goncharov, \textit{Moduli spaces of local systems and higher Teichm\"{u}ller theory,}
Publ. Math. Inst. Hautes \'{E}tudes Sci., 103 (2006) 1--211.
\bibitem {FP} S. Fomin, P. Pylyavskyy, \textit{Tensor diagrams and cluster algebras,}  arXiv:1210.1888.
\bibitem {FZ1}  S. Fomin, A. Zelevinsky, \textit{Cluster algebras. I. Foundations,} J. Amer. Math. Soc. 15 (2002), no. 2, 497--529.
\bibitem {FZ4}  S. Fomin, A. Zelevinsky, \textit{Cluster algebras. IV. Coefficients,} Compos. Math. 143 (2007), no. 1, 112--164.
\bibitem {GHKK} M. Gross, P. Hacking, S. Keel, M. Kontsevich, \textit{Canonical bases for cluster algebras,} arXiv:1411.1394.
\bibitem {GLSp} C. Geiss, B. Leclerc, J. Schr\"{o}er, \textit{Partial flag varieties and preprojective algebras,} Ann. Inst. Fourier (Grenoble) 58 (2008), no. 3, 825--876.
\bibitem {GLSk} C. Geiss, B. Leclerc, J. Schr\"{o}er, \textit{Kac-Moody groups and cluster algebras,} Adv. Math. 228 (2011), no. 1, 329--433.
\bibitem {GLSg} C. Geiss, B. Leclerc, J. Schr\"{o}er, \textit{Generic bases for cluster algebras and the Chamber ansatz,} J. Amer. Math. Soc. 25 (2012), no. 1, 21--76.
%\bibitem {Ka1}  V. G. Kac, \textit{Infinite root systems, representations of graphs and invariant theory I,} Invent. Math. 56 (1980), 57--92.
%\bibitem {Ka2}  V. G. Kac, \textit{Infinite root systems, representations of graphs and invariant theory II,} J. Algebra 78 (1982), 141--162.
\bibitem {Ke}  B. Keller, Quiver mutation in Java, http://www.math.jussieu.fr/~keller/quivermutation/.
\bibitem {Ki} A.D. King, \textit{Moduli of representations of finite-dimensional algebras,} Quart. J. Math. Oxford Ser. (2) 45 (1994), no. 180, 515--530.
\bibitem {KQ} Y. Kimura, F. Qin, \textit{Graded quiver varieties, quantum cluster algebras and dual canonical basis,} Adv. Math. 262 (2014), 261-312.
\bibitem {KT} A. Knutson, T. Tao, \textit{The honeycomb model of $\GL_n(\mb{C})$ tensor products. I. Proof of the saturation conjecture,} J.Amer. Math. Soc. 12 (1999), no. 4, 1055--1090.
%\bibitem {M1} G. Muller, \textit{Locally acyclic cluster algebras,} Adv. Math. 233 (2013), 207--247.
%\bibitem {M2} G. Muller, \textit{$A=U$ for locally acyclic cluster algebras,} SIGMA 10 (2014), 094, 8 pages.
%\bibitem {MS} G. Muller, D. E. Speyer, \textit{Cluster Algebras of Grassmannians are Locally Acyclic,} arXiv:1401.5137.
\bibitem {MSW} G. Musiker, R. Schiffer, L. Williams, \textit{Bases for cluster algebras from surfaces,} Compositio Math 149, 2, (2013) 217--263.
\bibitem {MPT} M. Herceg, M. Kvasnica, C.N. Jones, M. Morari, Multi-Parametric Toolbox (MPT), http://people.ee.ethz.ch/~mpt/.
\bibitem {P} P. Plamondon, \textit{Generic bases for cluster algebras from the cluster category,}  Int Math Res Notices (2013) 2013 (10): 2368--2420.
\bibitem {PVa} I. Pak, E. Vallejo, \textit{Combinatorics and geometry of Littlewood-Richardson cones,}  European J. Combin. 26 (2005), no. 6, 995--1008.
\bibitem {PV}  V. L. Popov, E. B. Vinberg, \textit{Invariant theory, in: Algebraic geometry. IV, Encyclopaedia of Mathematical Sciences,} vol. 55, Springer-Verlag, Berlin, 1994, 123--284.
\bibitem {R} A. N. Rudakov, \textit{Exceptional collections, mutations and helices.
Helices and vector bundles, 1--6,} London Math. Soc. Lecture Note Ser. 148, Cambridge Univ. Press, Cambridge, 1990.
\bibitem {S1} A. Schofield, \textit{Semi-invariants of quivers,} J. London Math. Soc. (2) 43 (1991), no. 3, 385--395.
\bibitem {S2} A. Schofield, \textit{General representations of quivers,} Proc. London Math. Soc. (3) 65 (1992), no. 1, 46--64.
\bibitem {SV} A. Schofield, M. Van den Bergh, \textit{Semi-invariants of quivers for arbitrary dimension vectors,} Indag. Math. (N.S.) 12 (2001), no. 1, 125--138.
\bibitem {Sc} J. Scott, \textit{Grassmannians and cluster algebras,} Proc. London Math. Soc. (3) 92 (2006), no. 2, 345--380.
\bibitem {SW} A. Skowro\'{n}ski, J. Weyman, \textit{The algebras of semi-invariants of quivers,} Transformation Groups. 5(4) (2000), 361--402.
\bibitem {T} T. Tran, \textit{$F$-polynomials in quantum cluster algebras,}  Algebr. Represent. Theory 14 (2011), no. 6, 1025--1061.
\end{thebibliography}

\end{document}